\documentclass{amsart}
\usepackage{amsmath}
\usepackage{amsthm}
\usepackage{amsfonts}
\usepackage[foot]{amsaddr}
\usepackage{amssymb}
\usepackage{bbm}
\usepackage{graphicx}
\usepackage{subcaption}
\usepackage{natbib}
\usepackage{enumitem}
\usepackage{tikz-cd}
\usepackage{comment}
\usepackage[text={440pt,575pt},headheight=9pt,centering]{geometry}
\usepackage[textwidth=2.8cm]{todonotes} 
\makeatletter
\newcommand{\DefineTodoColor}[2]{%
	\define@key{todonotes}{#1}[]{\gdef\@todonotes@currentbackgroundcolor{#2}}%
}
\makeatother

\usepackage{hyperref}
\hypersetup{
	colorlinks=true,
	linkcolor=blue,
	filecolor=magenta,
	urlcolor=cyan,
	citecolor=teal
}

\usepackage{tikz}
\usetikzlibrary{positioning}

\tikzset{
  node/.style = {circle, draw, thick, minimum size=8mm, inner sep=0pt},
  edge/.style = {-, line width=0.9pt}
}

\allowdisplaybreaks

\newcommand{\p}{\mathbf{p}}
\newcommand{\x}{\mathbf{x}}
\newcommand{\y}{\mathbf{y}}

\newcommand{\X}{\mathbf{X}}
\newcommand{\Y}{\mathbf{Y}}

\newcommand{\M}{\mathcal{M}}

\newcommand{\RR}{\mathbb{R}}
\renewcommand{\S}{\mathbb{S}}

\newcommand{\PP}{\mathbb{P}}

\newcommand{\CC}{\mathbb{C}}

\newcommand{\Var}{\operatorname{Var}}
\newcommand{\Det}{\operatorname{Det}}
\newcommand{\CM}{\operatorname{CM}}

\newcommand{\indep}{\perp \!\!\! \perp}
 
\newcommand{\bs}{\boldsymbol}

\newcommand{\HR}{H\"usler--Reiss }
\DeclareMathOperator*{\tr}{\operatorname{tr}}

\newcommand{\eci}{{\rm eCI}}

\newcommand{\emld}{{\rm eMLD}}
\newcommand{\mld}{{\rm MLD}}
\newcommand{\emlt}{{\rm eMLT}}

\newcommand{\ones}{\mathbf{1}}
\newcommand{\zeros}{\mathbf{0}}
\newcommand{\bfe}{\mathbf{e}}

\newcommand{\bfx}{\mathbf{x}}
\newcommand{\bfy}{\mathbf{y}}

\newcommand{\sgn}{\mathrm{sgn}}
\newcommand{\rank}{\mathrm{rank}}
\newcommand{\tw}{{\rm tw}}

\newtheorem{thm}{Theorem}[section]
\newtheorem{prop}[thm]{Proposition}
\newtheorem{cor}[thm]{Corollary}
\newtheorem{lemma}[thm]{Lemma}

\newtheorem{conj}[thm]{Conjecture}

\theoremstyle{definition}
\newtheorem{defi}[thm]{Definition}
\newtheorem{ex}[thm]{Example}

\begin{document}
	\title[Algebraic statistics of H\"usler--Reiss graphical models in multivariate extremes]{Algebraic statistics of H\"usler--Reiss graphical models \\ in multivariate extremes}
  \author{Carlos Am\'endola\textsuperscript{1}}
        \address{\textsuperscript{1}Institute of Mathematics, Technische Universität Berlin, Germany}
        \author{Jane Ivy Coons\textsuperscript{2}} 
        \address{\textsuperscript{2}Department of Mathematical Sciences, Worcester Polytechnic Institute, Worcester, MA, USA}
        \author{Alexandros Grosdos\textsuperscript{3}} 
        \address{\textsuperscript{3}Institute of Mathematics, University of Augsburg, Augsburg, Germany}
        \author{Frank R\"ottger\textsuperscript{4}}
        \address{\textsuperscript{4}Department of Applied Mathematics, University of Twente, Enschede, The Netherlands}

\begin{abstract}
    The field of extreme value statistics is concerned with modeling and predicting rare events.
    In a Hüsler--Reiss graphical model, a graph represents extremal conditional independence (CI) relations between random variables. These models are exponential families parameterized by a graph Laplacian and are considered the analogue of multivariate Gaussian models in the extremal setting. We study these models from the perspective of algebraic geometry. Translating the CI relations into polynomial constraints in the parameters, we define extremal CI ideals and find a determinantal representation of their generators. 
    In terms of parametric inference, we study the extremal maximum likelihood degree as the number of solutions to a conditionally negative definite matrix completion problem. We also define and analyze the extremal maximum likelihood threshold for Hüsler--Reiss graphical models, which provides a certificate for the existence of a surrogate MLE in terms of the dimensionality of the point configuration that realizes the underlying summary statistic as a Euclidean distance matrix. We highlight throughout many interesting similarities but also differences with respect to Gaussian graphical models.
\end{abstract}
        \maketitle

\section{Introduction}

Algebraic statistical methodology has proven remarkably suitable for tackling problems in 
Gaussian graphical modeling \citep{sullivant2018,handbook}. 
This is because of the combinatorial features of the models
and the natural translation of statistical statements into polynomial (in)equalities.
Indeed, graphical modeling has been established as one of the pillars of algebraic statistics and remains 
a very active field of research,
see for instance \citet{nested20, toric21,thirdmoments}, to name a few.
Exciting recent advances in multivariate extremes have managed to connect graphical models with the theory of extremes in statistics.
The breakthrough paper of \citet{EH2020} sparked a vibrant line of research that allows for both powerful methodology and interpretable modeling of extreme dependence structures.  
In this paper,
we propose a new direction of research that aims to bring in algebraic and geometric machinery to better understand extremal graphical models.  

A popular approach for multivariate extremes are threshold exceedances, 
where we consider a multivariate observation as extreme when at least one dimension exceeds a very high threshold.
Here an important parametric family is the \emph{Hüsler--Reiss distribution}. 
For $d$-dimensional random vectors, 
this exponential family is parameterized by weighted graph Laplacians on a $d$-vertex graph. 
The Hüsler--Reiss family is considered an extremal analogue of the multivariate Gaussian, 
where the weighted Laplacian is the analogue of the precision matrix of a Gaussian distribution.  
Via the Fiedler identity \citep{devriendt2020effective}, 
these Laplacian matrices are in bijection with $d \times d$ variogram matrices, 
see e.g.~\citet[Appendix A]{REZ2021}. 
A variogram matrix is a symmetric matrix $\Gamma$ whose entries correspond to squared Euclidean distances between the vertices of a $(d-1)$-simplex, and can thus also be interpreted as Euclidean distance matrices.

Crucial to developing an algebraic theory is the observation 
that the variogram matrix $\Gamma$ acts as an analogue of the covariance matrix in the Gaussian setting.
Recall that for a Gaussian random vector $\X$ with covariance matrix $\Sigma$, 
the conditional independence statement 
$\X_A\indep \X_B \mid \X_C$
for any disjoint subsets $A,B,C\subset [d]$
holds exactly when the submatrix $\Sigma_{A \cup C, B\cup C}$ has rank $\#C$, see \citet[Proposition 3.1.13]{DSS2009}.
This translates algebraically to polynomial equations in the entries $\sigma_{ij}$ of the covariance matrix, 
coming from all minors of size $\#C+1$.
Thus one defines the conditional independence ideal of a graphical model to be the ideal generated by the above matrix minor polynomials for all conditional independence statements coming from the graph.

The standard notion of conditional independence is not applicable for a H\"usler--Reiss vector $\Y$, as its support is not a product space.
\citet{EH2020} proposed an alternative notion, where for any disjoint $A,B,C\subset [d]$ we say that $\Y_A$ is \emph{extremal conditionally independent} of $\Y_B$ given $\Y_C$ ($\Y_A\perp_e \Y_B\mid \Y_C$) if the corresponding classical conditional independence statements hold for $\Y$ restricted to all positive halfspaces, see Definition~\ref{def:extremalCI} for details. 

Resembling the Gaussian, it has been shown that extremal conditional independence for H\"usler--Reiss distributions can be encoded parametrically \citep{HES2022,DER2026}. 
In particular, it was established by \citet[Proposition 3.3]{FJA2023}
that
\begin{equation*}
    Y_i\perp_e Y_j\vert \Y_C \; \Longleftrightarrow \; \det \begin{pmatrix}
        -\frac{1}{2}\Gamma_{\{i\}\cup C,\{j\}\cup C} & \mathbf{1}\\
        \mathbf{1}^\top  & 0\\
    \end{pmatrix}=0
\end{equation*}
for any singletons $i,j\in [d]$ and disjoint $C\subset [d]$.
We generalize this result for all extremal conditional independence statements of the general form $\Y_A \perp_e \Y_B \vert \Y_C$ in Theorem \ref{thm:CMMinors}.
In particular, we show that such conditional independence statements are equivalent to $\mathrm{rank}(\CM(\Gamma_{A \cup C, B \cup C})) = \#C + 1$, 
where the \emph{Cayley-Menger matrix} is defined as
$\CM(\Gamma):=\left(\begin{smallmatrix}
        -\Gamma/2 & \mathbf{1}\\
        \mathbf{1}^\top  & 0\\
    \end{smallmatrix}\right).$

The rank restrictions on the Cayley-Menger matrix readily translate to polynomial equations in the entries of $\Gamma$,
prompting the definition of the extremal conditional independence ideal of a graph~$G$
\[
\eci_G := \sum_{A \perp_e B \mid C} \langle (\#C + 2) \times (\#C + 2) \text{ minors of } \CM(\Gamma_{A \cup C, B \cup C}) \rangle.
\]
Interestingly, we show in Theorem \ref{thm:CMGenerators} that this ideal is already generated by  
\begin{equation*}
    \langle \det \CM(\Gamma_{A', B'}) \mid A' \subset A \cup C, \ B' \subset B \cup C, \ \#A' = \#B' = \#C + 1 \rangle.
    \end{equation*}

For Gaussian data,
the problem of maximum likelihood estimation in Gaussian graphical models 
can be reformulated as a positive-definite 
matrix completion problem \citep{dempster1972covariance}. 
A similar approach provides a surrogate maximum likelihood estimator in H\"usler--Reiss graphical models \citep{HES2022}.
Here, for data arising from a H\"usler--Reiss distribution (or, more realistically, in its domain of attraction),
one can obtain an empirical variogram matrix $\overline{\Gamma}$ \citep{EV2020}.
For a given undirected graph $G$,
the resulting matrix completion problem asks whether the partial matrix $\Gamma$ with entries $\Gamma_{ij} = \overline{\Gamma}_{ij}$ for all edges $(i,j)$ of $G$
admits a Euclidean distance matrix completion that is 
compatible with the model induced by $G$,
i.e., a completion such that the corresponding \HR precision matrix has zeros for all nonedges of $G$. 
This is a semialgebraic problem,
consisting of algebraic equations coming from the structural zeros
and inequalities arising from restricting to Euclidean distance matrices.

Another main contribution we make is studying the \emph{extremal maximum likelihood degree} of the Euclidean distance matrix completion problem, that is, the number of complex solutions of the underlying algebraic equations.
This model invariant is a direct analogue of the Gaussian maximum likelihood degree, first studied by \citet{sturmfels2010multivariate,GeometryML12}, and has led to a fruitful line of research in algebraic statistics (e.g.~\citet{coons2020maximum,michalek2021maximum,amendola2024maximum}). It turns out that models defined by chordal graphs have both ML degree and extremal ML degree one, and we provide a closed formula for the surrogate MLE in Theorem \ref{thm:chordalmld1}.
Furthermore, we show that the extremal ML degree factorizes for graphs that are formed by gluing together two graphs over a common clique. 
In examples, we see that the extremal ML degree can be smaller or equal to the Gaussian ML degree, and we conjecture (Conjecture \ref{conj:degree}) that this is always the case. 

Finally, we introduce an \emph{extremal maximum likelihood threshold}, 
giving rank conditions on the summary statistic such that a solution of the completion problem generically exists. This is an analogue of the classical maximum likelihood threshold for Gaussian graphical models \citep{gross2018maximum}.
This threshold is important as it describes conditions under which the maximum likelihood estimator exists almost surely. 
For a Gaussian distribution such rank conditions are with respect to the empirical covariance matrix, 
while for a H\"usler--Reiss distribution the extremal maximum likelihood threshold arises from constraints on an empirical variogram matrix.
The two cases seem to be closely related, 
since the bounds given by the maximum clique number and the treewidth in Theorem \ref{thm:ineqGaussian}, which are well known in the Gaussian case, 
hold analogously for the extremal case
by decreasing bounds by $1$, see Theorem~\ref{thm:ineqHuesler}.

The paper is structured as follows. In Section~\ref{sec:prelim} we recall the necessary background on variograms, extremal conditional independence, and Hüsler-Reiss distributions. Our first results characterizing the extremal conditional independence ideal, Theorems \ref{thm:CMMinors} and \ref{thm:CMGenerators}, are presented in Section~\ref{sec:CIideal}. In Section \ref{sec:MLdegree} we study the maximum likelihood estimation problem, giving an explicit description of the extremal MLE for chordal graphs in Theorem \ref{thm:chordalmld1} and proving the multiplicativity of the extremal ML degree in Theorem \ref{thm:multiplicative}. Extremal maximum likelihood thresholds are the topic of Section~\ref{sec:mlt}, where we give bounds in Theorem \ref{thm:ineqHuesler} and an algebraic criterion in Theorem \ref{thm:zeroidealcriterion}. In the Appendix we collect some useful properties of variogram matrices.

\section{Preliminaries}
\label{sec:prelim}

We begin our introduction to H\"usler-Reiss graphical models by reviewing fundamental facts and definitions about Laplacian matrices, variogram matrices (or Euclidean distance matrices), and the relationship between the two of them. In Section \ref{sec:ExtremalCI}, we introduce multivariate Pareto distributions and define a notion of \emph{extremal conditional independence} in this general setting. Finally in Section \ref{sec:HR} we define the central objects of study in this work, H\"usler-Reiss graphical model. These are examples of multivariate Pareto distributions that are exponential families parameterized by a graph Laplacian which makes them amenable to study through an algebraic lens.

\subsection{Variograms and Laplacians}

We denote by $\S^d$ the set of all real symmetric $d\times d$ matrices, $\S_0^d\subset\S^d$ the subset of matrices with zero diagonal. We use $\S^d_{\succeq}\subseteq\S^d$ to denote the subset of positive semidefinite matrices and $
\S^d_{\succ}$ for the positive definite matrices. 
Let $G=(V,E)$ be a simple undirected graph with vertex set $V$ and edge set $E\subset V\times V$, where typically $V=[d]=\{1,\ldots,d\}$.
We endow the graph $G$ with a weighted adjacency matrix $Q\in \S^{d}_0$, such that for each $(i,j)\not\in E$, the $ij$th entry $Q_{ij}=0$.
Let $D$ be the $d\times d$ diagonal matrix with $D_{ii}=\sum_{i\neq j}Q_{ij}$.
We call $\Theta=D-Q$ a \emph{signed Laplacian matrix} and denote by $\mathbb{U}^d=\{\Theta\in \S^d_\succeq:\Theta\ones=\mathbf{0}\}$
the set of all positive semidefinite signed Laplacians.
We denote by $\Sigma=\Theta^+$ the Moore--Penrose pseudoinverse of the signed Laplacian $\Theta$.
Let $\mathbf{p}=\frac{1}{2}\Theta d_\Sigma+\frac{1}{d}\mathbf{1}$ and $R^2=\frac{1}{4}d_\Sigma^\top\Theta d_\Sigma +\frac{1}{d}\mathbf{1}^\top d_\Sigma$, where $d_\Sigma$ denotes the diagonal of $\Sigma$ in vector notation.
The \emph{Fiedler--Bapat identity} \citep{DevriendtPhD} states that
\begin{equation}\label{eq:FiedlerBapat}
\begin{pmatrix}
    \Theta &\p\\
    \p^\top & R^2\\
\end{pmatrix}^{-1} = 
    \begin{pmatrix}
        -\Gamma/2 & \mathbf{1}\\
        \mathbf{1}^\top & 0\\ 
    \end{pmatrix} =:\CM(\Gamma)
\end{equation}
where $\Gamma$ is a Euclidean distance or variogram matrix.  Given a configuration of $d$ points in a Euclidean space, their  \emph{Euclidean distance matrix} is the symmetric $d \times d$ matrix whose $ij$ entry is the squared Euclidean distance between the $i$th and $j$th points. These are also known as \emph{variogram} matrices. In the present work, we use the terms interchangeably, typically opting for ``variogram matrix" but also using ``Euclidean distance matrix" when we want to emphasize that the matrix comes from such a point configuration. 
For an equivalent characterization of Euclidean distance matrices, we introduce the cone of conditionally negative definite matrices.
First, a matrix is \emph{strictly conditionally negative definite} if it belongs to the set
\[
\mathcal{C}^d:=\{\Gamma\in\S^d_0: v^\top \Gamma v <0  \text{ for all } v \neq \zeros \text{ such that } v \perp \ones \}.\]
Matrices in the closure of this cone, $\bar{\mathcal{C}}^d$, which satisfy $v^\top \Gamma v \leq 0$ for all such $v$, are called \emph{conditionally negative definite} (CND). The Schoenberg criterion 
\citep{schoenberg1935, young1938discussion} states that a matrix is conditionally negative definite if and only if it is a Euclidean distance matrix of a point configuration.

Let $\theta (\Gamma)=\Theta$ denote the map which obtains the signed Laplacian matrix $\Theta$ from $\Gamma$. 
The map $\theta$ gives a bijection between the set of strictly conditionally negative definite matrices $\mathcal{C}^d$ and the set of rank $d-1$ positive semidefinite signed Laplacian matrices of connected graphs.
A CND variogram matrix $\Gamma$ can also be obtained directly from $\Sigma$, the Moore--Penrose pseudoinverse of $\Theta$, via the map 
\begin{equation}\label{eq:gammaofsigma}
    \gamma(\Sigma)=d_\Sigma \ones^\top + \ones d_\Sigma^\top -2 \Sigma.
\end{equation}
Its inverse map is 
\begin{equation}
\label{eq:sigmaofgamma}
    \sigma(\Gamma) = (I-\frac{1}{d}\ones\ones^\top)\left(-\frac{\Gamma}{2}\right)(I-\frac{1}{d}\ones\ones^\top),
\end{equation}
see also \citet[Appendix~A]{REZ2021}. We note that since the vector $\ones$ belongs to the kernel of the projection matrix $I - \frac{1}{d} \ones \ones^\top$, we have that $\ones \in \ker(\sigma(\Gamma))$ for every variogram matrix $\Gamma$.

It is also occasionally useful to consider the matrix $\Theta^{(k)}$, which is the matrix obtained from $\Theta$ by deleting the $k$th row and column. Since $\Theta$ has rank $d-1$, it is invertible and we define $\Sigma^{(k)} = \big(\Theta^{(k)})^{-1}$. We can also recover $\Sigma^{(k)}$ from $\Gamma$ via the \emph{covariance mapping}
\begin{equation}\label{eq:Sigma^kFromGamma}
\Sigma^{(k)}_{ij} = \frac{1}{2} (\Gamma_{ik} + \Gamma_{jk} - \Gamma_{ij}),
\end{equation}
see \citet[Equation~(9)]{REZ2021}.

\begin{ex}\label{ex:path}
Let
    \begin{align*}
					\Gamma = \begin{pmatrix}
						0&\Gamma_{12}&\Gamma_{13}\\
						\Gamma_{12}&0&\Gamma_{23}\\
						\Gamma_{13}&\Gamma_{23}&0\\
					\end{pmatrix}=\begin{pmatrix}
					0&9&25\\
					9&0&16\\
					25&16&0\\
				\end{pmatrix}
			\end{align*}
be the Euclidean distance matrix of the 2-simplex in Figure~\ref{fig:path}.
The Fiedler--Bapat identity in Equation \eqref{eq:FiedlerBapat} takes the form
\begin{align*}
\renewcommand{\arraystretch}{1.5}
\CM(\Gamma)^{-1} = \left(
\begin{array}{ccc|c}
	0 & -\frac{9}{2} & -\frac{25}{2} & 1 \\
	-\frac{9}{2} & 0 & -8 & 1 \\
	-\frac{25}{2} & -8 & 0 & 1 \\
	\hline
	1 & 1 & 1 & 0 \\
\end{array}
\right)^{-1}
&=
\renewcommand{\arraystretch}{1.5}
\left(
\begin{array}{ccc|c}
\frac{1}{9} & -\frac{1}{9} & 0 & \frac{1}{2} \\
-\frac{1}{9} & \frac{25}{144} & -\frac{1}{16} & 0 \\
0 & -\frac{1}{16} & \frac{1}{16} & \frac{1}{2} \\
\hline
\frac{1}{2} & 0 & \frac{1}{2} & \frac{25}{4} \\
\end{array}
\right).
\end{align*}
Note that the upper left $3\times 3$-block $\Theta$ is the signed Laplacian matrix for the path graph in Figure~\ref{fig:path} with $Q_{12}=\frac19$ and $Q_{23}=\frac{1}{16}$. The Moore--Penrose pseudoinverse of $\Theta$ is
\[
\Sigma = \Theta^+ =  \frac{1}{9}\begin{pmatrix}
    52 & -2 & -50 \\
    -2 & 25 & -23 \\
    -50 & -23 & 73
\end{pmatrix}.
\]
We see that applying the map $\gamma$ to $\Sigma$ indeed recovers $\Gamma = \gamma(\Sigma)$. 
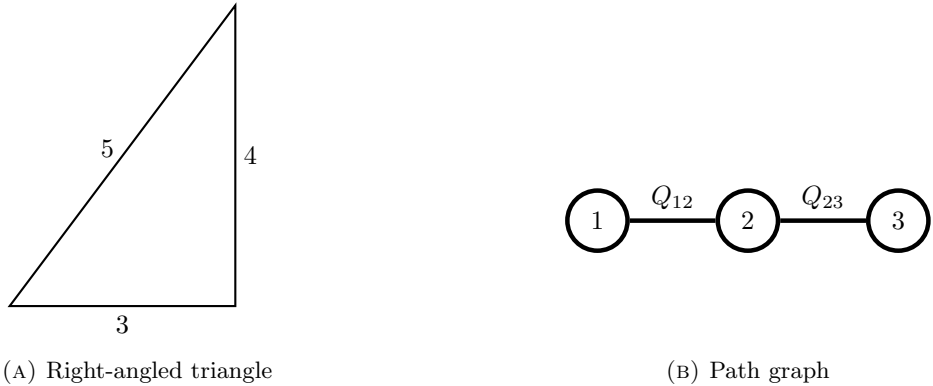
\begin{figure}[ht]
    \centering
    \begin{subfigure}[b]{0.5\textwidth}
        \centering
        \begin{tikzpicture}
            \coordinate (A) at (0,0);
            \coordinate (B) at (3,0);
            \coordinate (C) at (3,4);
            
            \draw[thick] (A) -- (B) -- (C) -- cycle;
            
            \node at (1.5, -0.25) {3}; 
            \node at (1.3, 2.1) {5};   
            \node at (3.2, 2) {4};     
        \end{tikzpicture}
        \caption{Right-angled triangle}
    \end{subfigure}
    \hfill
    \begin{subfigure}[b]{0.45\textwidth}
        \centering
        \begin{tikzpicture}[line width=.6mm]
            \node[minimum size=8mm,shape=circle,draw=black] (A) at (0, 1.5) {1};
            \node[minimum size=8mm,shape=circle,draw=black] (B) at (2, 1.5) {2};
            \node[minimum size=8mm,shape=circle,draw=black] (C) at (4, 1.5) {3};
            \node[] (D) at (0,0) {};
             \draw (A) -- (B) node[midway, above] {$Q_{12}$};
            \draw (B) -- (C) node[midway, above] {$Q_{23}$};
        \end{tikzpicture}
        \caption{Path graph}
    \end{subfigure}
    \caption{Right-angled triangle corresponding to $\Gamma$, and path graph corresponding to the adjacency matrix $Q$, see Example~\ref{ex:path}.}\label{fig:path}
\end{figure}
\end{ex}

Given a $d \times d$ Euclidean distance matrix $\Gamma$, suppose that $\sigma(\Gamma) = \Sigma$ has rank $m$. Then we can decompose $\Sigma = B^\top B$ where $B$ is an $m \times d$ matrix. 
As a consequence of Equation \eqref{eq:gammaofsigma}, the $d$ columns of $B$ yield a point configuration whose pairwise squared Euclidean distances are the entries of $\Gamma$; 
$\Sigma$ is called the \emph{Gram matrix} of this point configuration. 

\begin{defi}
    Let $\Gamma$ be a $d \times d$ Euclidean distance matrix. The \emph{dimensionality} of $\Gamma$ is the minimum $m$ such that $\Gamma$ is the Euclidean distance matrix of $d$ points in $m$ dimensions. In other words, it is the rank of $\sigma(\Gamma)$.
\end{defi}

Gower characterized the dimensionality of $\Gamma$ in terms of its rank and the geometry of the point configuration that realizes it \cite[Theorem~6]{gower1985properties}. In particular, if $\rank(\Gamma) = r$, then its dimensionality is either $r-2$ or $r-1$. In Proposition~\ref{prop:RankBoundOnGamma}, we show that when $r < d$, a generic $\Sigma$ of rank $r-2$ yields a variogram matrix $\gamma(\Sigma)$ of rank $r$.

\subsection{Extremal conditional independence}\label{sec:ExtremalCI}
In this section, we introduce multivariate Pareto distributions in the context of extreme value statistics. We summarize the statistical foundations for extremal conditional independence in this general setting before specializing to the H\"usler-Reiss setting in Section \ref{sec:HR}. 

Let $\X$ be a $d$-variate random vector, where each margin $X_i$ follows an absolutely continuous probability distribution with cumulative distribution function $ F_i$.
The bivariate extremal correlation coefficient
$\chi_{ij}:=\lim_{q\to 0}\PP(F_i(X_i)>1-q\mid F_j(X_j)>1-q),$
assuming that the limit exists, provides a characterization of extremal dependence.
When $\chi_{ij}>0$ for all pairs $i,j\in[d]$ we call $\X$ asymptotically dependent.

In extremal dependence modeling it is commonly assumed that $\X$ has standardized margins, i.e.~that all margins follow the same univariate distribution.
This simplifies dependence modeling, as the margins do not influence the underlying dependence structure.
Thus, in this paper we assume that all $X_i$ follow the standard exponential distribution.
A popular approach in multivariate extreme value theory is to consider values as extreme when they exceed some high threshold.
This can be formalized through a limiting distribution for threshold exceedances $\X-u\ones$, where $\ones^\top = (1,\ldots,1)$, under the condition that at least one component of $\X$ exceeds the threshold $u>0$.
Under the assumption of asymptotic dependence, the limit
\begin{align}
\PP(\Y\le \y):=\lim_{u\to \infty}\PP(\X-u\ones \le \y\mid \X\not\le u\ones )   \label{eq:multivariate_pareto}
\end{align}
of threshold exceedances gives rise to a probability distribution $\PP(\Y\le \y)$ where the support of $\Y$ is $\mathcal{L}:=\{\y\in\RR^d:\; \y\not\le \mathbf{0}\}$.
This distribution is typically called a (generalized) multivariate Pareto distribution \citep{roo2006}.
Any multivariate Pareto distribution is characterized by an exponent measure $\Lambda$  defined on $[-\infty,\infty)^d\setminus{-\boldsymbol{\infty}}$, such that
\[\PP(\Y\le \y)=\frac{\Lambda^c(\y\wedge\mathbf{0})-\Lambda^c(\y)}{\Lambda^c(\mathbf{0})},\]
where $\wedge$ denotes the componentwise minimum and $\Lambda^c(\y):=\Lambda([-\infty,\infty)^d\setminus [-\boldsymbol{\infty},\y])$.
If $\Lambda$ is assumed to be absolutely continuous, it allows a Lebesgue density $\lambda$, such that the density of $\Y$ is $f_Y(\y)=\lambda(\y)/\Lambda(\mathcal{L})$.
For any halfspace $\mathcal{L}^k:=\{\y\in\mathcal{L}: y_k\ge 0\}$ it is $\Lambda(\mathcal{L}^k)=1$, and $\Y^k:=\Y|\{ Y_k>0 \} $ has probability density $\lambda$ restricted to its support $\mathcal{L}^k$.

For some $C\subseteq[d]$ let $\lambda_C(\y_C):=\int_{\RR^{d-\#C}}\lambda(\y_{\setminus C})d\y_{\setminus C}. $ 
The marginal random vector $\Y^k_C$ has probability density
$\lambda_C(\y_C)$ on its support, for all $k\in C$.
Furthermore, $\lambda_C$ is the exponent measure density of the threshold exceedance limit of $\X_C$ in \eqref{eq:multivariate_pareto}. We denote by $\Y_{(C)}$ the resulting multivariate Pareto vector.
Usually, $\Y_{(C)}$ is not equal in distribution to the marginal $\Y_C$.

\citet{EH2020} introduce a notion of conditional independence for multivariate Pareto distributions.
Under the assumption that $\Lambda$ admits a density $\lambda$, or equivalently that $\X$ is asymptotically dependent, this is defined as follows:
\begin{defi}\label{def:extremalCI}
   Let $\Y$ be multivariate Pareto with exponent measure density $\lambda$ and $A,B,C$ disjoint subsets of $[d]$.
   We say that $\Y_A$ is \emph{extremal conditional independent} of $\Y_B$ given $\Y_C$ if 
\begin{align}
   \lambda_{ABC}(\y_{ABC})\lambda_C(\y_C)=\lambda_{AC}(\y_{AC})\lambda_{BC}(\y_{BC}) \label{eq:extremalCI} 
\end{align}   
for all $\y\in\mathcal{L}$. We abbreviate this statement as $\Y_A\perp_e \Y_B \mid \Y_C$.
\end{defi}
Since the restriction of $\lambda$ to a halfspace $\mathcal{L}^k$ is the probability density of $\Y^k$, this means that $\Y_A\perp_e \Y_B \mid \Y_C$ is equivalent to $$ \Y_A^k\indep \Y_B^k \mid \Y_C^k, \quad \forall k \in C.$$
Let $\lambda(\y_{A}|\y_C):=\frac{\lambda_{AC}(\y_{AC})}{\lambda_C(y_C)}$ be the conditional exponent measure density. Extremal conditional independence can also be rephrased as factorization of the conditional exponent measure density $\lambda(\y_{AB}|\y_C)$, namely Equation \eqref{eq:extremalCI} can be written as
\begin{equation}\label{eq:fact}  \lambda(\y_{AB}|\y_C)=\lambda(\y_{A}|\y_C)\lambda(\y_{B}|\y_C).
\end{equation}

Among the collection of all multivariate Pareto distributions, the parametric family of \HR distributions is particularly attractive for extremal dependence modeling, as it shares many properties of the multivariate Gaussian. In the next section, we describe how \HR models can be parameterized by a variogram or Laplacian matrix and outline an algebraic interpretation of pairwise conditional independence in this setting.

\subsection{H\"usler--Reiss distributions}\label{sec:HR}
H\"usler--Reiss distributions form a parametric subfamily of multivariate Pareto distributions.
In multivariate extremes they are considered as an analogue of the multivariate Gaussian distribution.
Under the assumption of asymptotic dependence, they are parameterized by conditionally negative definite variogram matrices $\Gamma \in \mathcal{C}^d  $ and have an exponent measure density
\[\lambda(\y)=c_1 \exp\left(-\frac{1}{2}\begin{pmatrix}
    \y^\top  & 1\\
\end{pmatrix} \CM(\Gamma)^{-1}\begin{pmatrix}
    \y\\
    1\\
\end{pmatrix} \right), \]
where $c_1=\frac{1}{\sqrt{-\det(\CM(2\pi\Gamma))}} = \left(\det\left(\begin{smallmatrix}
        -\pi \Gamma & \mathbf{1}\\
        -\mathbf{1}^\top  & 0\\
    \end{smallmatrix}\right)\right)^{-1/2}$ \citep{HES2022,EGR2025}.
A H\"usler--Reiss marginal $\Y_{(C)}$ is again H\"usler--Reiss with parameter $\Gamma_{C,C}$.
The conditional exponent measure densities are in fact Gaussian densities.
For any disjoint $A, C \subset V$ it is
    \begin{align*}
      \lambda(\y_A\mid \y_{C}) &=  \frac{1}{\sqrt{\det(2\pi\Sigma^*)}}\exp\left(-\frac{1}{2} (\y_A-\mu^*)^{\top}(\Sigma^*)^{-1}(\y_A-\mu^*) \right),
    \end{align*}
    where 
    \begin{align*}
      \Sigma^*&=-\frac{1}{2}\Gamma_{A,A} -\begin{pmatrix}
        -\frac{1}{2}\Gamma_{A,C}& \ones\\
      \end{pmatrix}
      \CM(\Gamma_{C,C})^{-1}
      \begin{pmatrix}
        -\frac{1}{2}\Gamma_{C,A}\\
        \ones^{\top}\\
      \end{pmatrix}
    \end{align*} 
    and
    \begin{align*}
      \mu^*=\begin{pmatrix}
        -\frac{1}{2}\Gamma_{A,C} & \ones\\    \end{pmatrix}\CM(\Gamma_{C,C})^{-1}\begin{pmatrix}
        \y_{C}\\
        1\\
      \end{pmatrix}.
    \end{align*}

The Gaussianity of $\lambda(\y_A\mid \y_{C})$ explains some of the appealing properties of the \HR{} distribution.
For example, it was shown in \citet{FJA2023} that
\[Y_i\perp_e Y_j\vert \Y_C \Longleftrightarrow \det \begin{pmatrix}
        -\frac{1}{2}\Gamma_{\{i\}\cup C,\{j\}\cup C} & \mathbf{1}\\
        \mathbf{1}^\top  & 0\\
    \end{pmatrix}=0, \]
i.e.~that extremal conditional independence between two singletons given a non-empty set is equivalent to a vanishing almost-principal minor of the Cayley--Menger matrix $\CM(\Gamma)$.
This in particular implies the equivalence
\[Y_i\perp_e Y_j\vert \Y_{V\setminus\{i,j\}} \Longleftrightarrow \Theta_{ij}=0,\]
which was first shown by \citet{HES2022}. This means that the pairwise Markov property can be characterized through vanishing entries in the \HR precision matrix $\Theta$, which has led to new parametric methodology for undirected extremal graphical models. We refer to the review article of \citet{EHLR2024} for further details.

We can therefore define any \HR graphical model as a geometric object as follows.
\begin{defi}
    Let $G=([d],E)$ be an undirected graph. The \emph{Hüsler--Reiss variogram model associated to $G$} is  
    \begin{equation}\label{eq:HRmodel}
        \mathcal{M}_G = \{ \Gamma\in\mathcal{C}^d: \theta(\Gamma)_{ij}=0  \text{ for all } (i,j) \notin E \} \subset \mathcal{C}^d,
    \end{equation}
    that is, the set of variogram matrices arising from the H\"usler--Reiss distributions Markov to $G$.
\end{defi}

In the context of regular conditional independence, random vectors with positive densities (such as the Gaussian), global and pairwise Markovianity are equivalent \citep[Theorem 13.1.4]{sullivant2018}. 
Because extremal conditional independence of a multivariate Pareto vector $\Y$ can be characterized via regular conditional independence of the random vectors $\Y^k$, $k\in [d]$, (extremal) pairwise and global Markovianity are also equivalent for a H\"usler-Reiss distribution with a positive exponent measure density. The extremal global Markov property states that if $A, B$ and $C$ are disjoint subsets of $[d]$ such that removing $C$ from $G$ separates the vertices in $A$ from the vertices in $B$, then the extremal conditional independence statement $\mathbf{Y}_A \perp_e \mathbf{Y}_B \mid \mathbf{Y}_C$ holds.

The variogram matrix of a \HR distribution can be obtained for any $k\in [d]$ from
\begin{equation}
    \Gamma_{ij}=\Var[Y^k_i-Y^k_j]. \label{eq:population_vario}
\end{equation}
This representation gives rise to an estimator for $\Gamma$.
For i.i.d.~observations $\y_1,\ldots,\y_n \in \mathcal{L}$, the empirical variogram \citep{EV2020} averages sample versions of \eqref{eq:population_vario} in all canonical halfspaces within the support of $\Y$.
For each $k\in[d]$, let $J_k:=\{\ell \in [n]:\y_\ell \in \mathcal{L}^k \}$.
The empirical variogram for a \HR distribution is given as
\begin{equation}
\overline{\Gamma}_{ij}=\frac{1}{d}\sum_{k=1}^d\widehat{\Var}\left(y_{\ell,i}-y_{\ell, j}: \ell \in J_k\right), \label{eq:emp_vario}
\end{equation}
where $\widehat{\Var}$ denotes the sample variance. 
This provides a consistent estimator of the variogram, and can be considered as a summary statistic in a \HR model.

In practice, we can not observe the limiting model $\Y$ directly, but only the underlying data-generating process $\X$.
Thus, to obtain an approximate sample for $\Y$, it is customary to threshold the data using a high quantile of the standardized margins of $\X$.
In this approach the approximate sample for $\Y$ is the collection of all $\x_i-u\ones$ where $\max \x_i >u$, where $u$ is some high quantile of the standard exponential distribution.
Here, the choice of the threshold can be seen as a hyperparameter which influences the tradeoff between sample size and accuracy of the approximate sample.

\section{The Extremal Conditional Independence Ideal}
\label{sec:CIideal}

The fact that $\Theta$ is a signed Laplacian matrix for $G$ enforces linear constraints on $\Theta$; in particular, $\Theta$ has a sparsity pattern corresponding to non-edges of $G$ and all of its row sums are $0$. Via the Fiedler-Bapat identity (Equation \eqref{eq:FiedlerBapat}), we see that this in turn enforces \emph{polynomial} constraints on the variogram matrix $\Gamma$. In this section, we consider these algebraic constraints in terms of graph separation.
Algebraically, one can associate a polynomial ideal in the variables $\gamma_{ij}$ to $\mathcal{M}_G$:
$$I_G := \mathcal{I}(\M_G) \subset \CC[\gamma_{ij} \mid 1 \leq i < j \leq n],$$
consisting of all polynomials that vanish on $\M_G$. We call this the \emph{extremal ideal} associated to $G$.

\begin{ex}
    Consider a trivariate H\"usler--Reiss model with variogram matrix $\Gamma$ that is Markov to the path graph from Figure~\ref{fig:path}(B).  
    From the Fiedler--Bapat identity (Equation \eqref{eq:FiedlerBapat}) we obtain 
    \[\theta(\Gamma)_{13}=\frac{1}{\det(\CM(\Gamma))}(\Gamma_{13}-\Gamma_{12} - \Gamma_{23} ).\]
    In fact, we compute that its extremal ideal is the principal linear ideal
    $$I_G = \langle \gamma_{12} + \gamma_{23} - \gamma_{13} \rangle \subset \CC[\gamma_{12},\gamma_{13},\gamma_{23}]. $$
    If we interpret the matrix $\Gamma$ as a Euclidean distance matrix, by the converse of the Pythagorean theorem, we conclude that the associated 2-simplex must be a right-angled triangle, compare Example~\ref{ex:path} and Figure~\ref{fig:path}. 
\end{ex}

By focusing on conditional independence, there is another natural ideal that we can associate to $G$.
A key fact in algebraic statistics is that for multivariate Gaussian distributions, conditional independence statements are equivalent to algebraic constraints on the covariance matrix. Namely, if $\X \sim \mathcal{N}(\mu,\Sigma)$, then 
\begin{align}
    \X_A \indep \X_B \, | \,\X_C \iff \mathrm{rank}(\Sigma_{A\cup C, B\cup C})=\#C. \label{eq:rankCI}
\end{align}

Previously, the only known parametric description of extremal conditional independence statements $\Y_A\perp_e \Y_B\mid \Y_C$ for H\"usler--Reiss distributions where $A,B$ are not singletons was via modularity of a multiinformation-inspired set function \citep{DER2026}.
We now show that there is an extremal analogue to \eqref{eq:rankCI} for Hüsler--Reiss distributions. 
\begin{thm}\label{thm:CMMinors}
    Let $\Y$ be a \HR random vector with variogram $\Gamma$ and let $A, B, C \subset [d]$ be disjoint. Then $\Y_A$ and $\Y_B$ are extremal conditionally independent given $\Y_C$ if and only if
    \[
    \mathrm{rank}(\CM(\Gamma_{A \cup C, B \cup C})) = \#C + 1.
    \]
\end{thm}

\begin{proof}
    By marginalizing, we may assume without loss of generality that $A\cup B \cup C =[d]$. Recall from Equation \eqref{eq:fact} that $\Y_A\perp_e \Y_B \mid \Y_C$ is equivalent to $\lambda(\y_{AB}|\y_C)=\lambda(\y_A|\y_C)\lambda(\y_{B}|\y_C).$
    Note also that $\Y^k_{D}\mid\{\Y^k_C=\y_C\} \sim \lambda(\y_{D}|\y_C)$ for any nonempty $D\subseteq [d]\setminus C$ and $k\in C$.
    Hence, these are Gaussian random vectors.
    Thus we are searching a factorization of a $\#(A\cup B)$-variate Gaussian random vector with covariance 
    \[-\frac{1}{2}\Gamma_{AB,AB}-\begin{pmatrix}
        -\frac{1}{2}\Gamma_{AB,C} & \bs 1\\
    \end{pmatrix}\CM(\Gamma_{C,C})^{-1}\begin{pmatrix}
        -\frac{1}{2}\Gamma_{C,AB} \\
        \bs 1^\top \\
    \end{pmatrix}.\]
    Therefore, $\Y_A\perp_e \Y_B \mid \Y_C$ if and only if
    \begin{align}
        -\frac{1}{2}\Gamma_{A,B}-\begin{pmatrix}
        -\frac{1}{2}\Gamma_{A,C} & \bs 1\\
    \end{pmatrix}\CM(\Gamma_{C,C})^{-1}\begin{pmatrix}
        -\frac{1}{2}\Gamma_{C,B} \\
        \bs 1^\top \\
    \end{pmatrix}=\mathbf{0}.\label{eq:AB-block}
    \end{align}
    Note that the left hand side of \eqref{eq:AB-block} is the Schur complement of the $-\frac{1}{2}\Gamma_{A,B}$ block in $\CM(\Gamma_{AC,BC})$. As $\CM(\Gamma_{C,C})$ is by assumption invertible, by the Guttman rank additivity formula we have that $$\operatorname{rank}(\CM(\Gamma_{AC,BC}))=\#C+1$$ if and only if \eqref{eq:AB-block} holds.    
\end{proof}
This characterization of extremal conditional independence allows us to define the \emph{extremal conditional independence ideal} of the model in terms of vanishing minors of $\CM(\Gamma)$.

\begin{defi}
    The \emph{extremal conditional independence ideal} of $G$ is the ideal
    \[
    \eci_G := \sum_{A \perp_e B \mid C} \langle (\#C + 2) \times (\#C + 2) \text{ minors of } \CM(\Gamma_{A \cup C, B \cup C}) \rangle\subset \CC[\gamma_{ij} \mid 1 \leq i < j \leq n].
    \]
    In words, this is the ideal of all minors of $\CM(\Gamma)$ that vanish according to separation statements in $G$, as guaranteed by Theorem \ref{thm:CMMinors} and the global Markov property. 
\end{defi}

In fact, not all of these minors are required to generate the extremal conditional independence ideal. In Theorem \ref{thm:CMGenerators}, we show that the determinants that are of Cayley-Menger type suffice as a generating set. We establish this using the following lemmas.

\begin{lemma}\label{lem:Dj}
    Let $M = (m_{ij})_{i,j=1}^n$ be an $n \times n$ matrix. For each $j$, let $D_j$ be the determinant of the matrix obtained from $M$ by replacing the $j$th column with the column of all ones. Then
    \[
    \det(M) = \sum_{j=1}^n m_{1j}D_j.
    \]
\end{lemma}

\begin{proof}
    First, recall that
    \[
    \det(M) = \sum_{\tau \in S_n} \sgn(\tau) \prod_{i=1}^n m_{\tau(i) i} \qquad \text{and} \qquad 
    D_j = \sum_{\tau \in S_n} \sgn(\tau) \prod_{\substack{i=1 \\ i \neq j}}^n m_{\tau(i) i}.
    \]
    The terms of $m_{1j} D_j$ that are also terms of $\det(M)$ correspond to the permutations $\tau \in S_n$ such that $\tau(j) = 1$. Hence every term of $\det(M)$ is a term in exactly one of the polynomials $m_{1j} D_j$. It remains to show that the other terms of these polynomials cancel with one another.

    Fix $j$ and a permutation $\tau \in S_n$ such that $\tau^{-1}(1) = k$ where $k \neq j$. The term of $m_{1j} D_j$ corresponding to $\tau$ is 
    \[
    \sgn(\tau) m_{1j} m_{1k} \prod_{\substack{i=1 \\ i \neq j,k}}^n m_{\tau(i) i}.
    \]
    Now consider the term of $m_{1k} D_k$ corresponding to the permutation $\tau' = (j \ k) \tau$.
    This permutation satisfies: $\sgn(\tau') = -\sgn(\tau)$, $\tau'(j) = 1$, $\tau'(k) = \tau(j)$ and $\tau'(i) = \tau(i)$ for all $i \neq j,k$. So the corresponding term of $m_{1k}D_k$ is
    \[
    -\sgn(\tau) m_{1j}m_{1k} \prod_{\substack{i =1 \\ i\neq j,k}}^n m_{\tau(i)i}.
    \]
    Finally, $(j \ k) \tau' = \tau$. Hence this establishes a sign-reversing involution between the terms of the polynomials $m_{1j} D_j$ that are not terms of $\det(M)$. So these terms cancel in the expression $\sum_{j=1}^n m_{1j} D_j$ and this sum is equal to $\det(M)$.  
\end{proof}

\begin{lemma}\label{lem:Dij}
    Let $M$ and $D_j$ be as in Lemma \ref{lem:Dj}. For each $i \in [n]$, let $D_{ij}$ be the determinant of the matrix whose $k\ell$ entry is
    \[
    \begin{cases}
        m_{k\ell} & \text{ if }  k \neq i \text{ and } \ell \neq j \\
        1 & \text{ if }  k = i \text{ and } \ell \neq j \\
        1 & \text{ if }  k \neq i \text{ and } \ell = j, \text{ and} \\
        0 & \text{ if }  k  = i \text{ and } \ell = j. \\
    \end{cases}
    \]
    Fix some $j \in [n]$ and some $x \in [n] \setminus \{j\}$.
    Then $D_j = \sum_{i=1}^n m_{ix} D_{ij}$.
\end{lemma}

\begin{proof}
    Without loss of generality, suppose $j > 1$ and let $x = 1$.
    We have that
    \[
    D_j = \sum_{\tau \in S_n} \sgn(\tau) \prod_{\substack{k=1 \\ k \neq \tau^{-1}(j)}}^n m_{k\tau(k)} \qquad \text{and} \qquad D_{ij} = \sum_{\substack{\tau \in S_n \\ \tau^{-1}(j) \neq i}} \sgn(\tau) \prod_{\substack{k=1 \\ k \neq i, \tau^{-1}(j)}}^n m_{k\tau(k)}. \]
    Let $\tau \in S_n$ and let $i = \tau^{-1}(1)$. Since $j \neq 1$, we have that $i \neq \tau^{-1}(j)$; so $\tau$ has a corresponding term in both $D_j$ and $m_{i1}D_{ij}$. This term is
    \[
    \sgn(\tau) m_{i1} \prod_{\substack{k=1 \\ k \neq i, \tau^{-1}(j)}}^n m_{k\tau(k)}.
    \]
    Note that this is not a term of any other $m_{s1}D_{sj}$ for $s \neq i$ since $m_{s1}$ does not divide it. So it remains to show that all other terms of the polynomials $m_{i1}D_{ij}$ cancel with each other.

    Fix some $i \in [n]$ and let $\tau$ be a permutation such that $i \neq \tau^{-1}(1)$. Suppose instead that $\tau^{-1}(1) = s$ for some $s \neq i$. Assume further that $\tau^{-1}(j) \neq i$ so that $\tau$ corresponds to a term of $D_{ij}$. Note that since $\tau^{-1}(1) = s$ and $j \neq 1$, we also have that $\tau^{-1}(j) \neq s$. So the corresponding term of $m_{i1}D_{ij}$ is
    \[
    \sgn(\tau) m_{i1}m_{s1} \prod_{\substack{k=1 \\ k \neq i,s,\tau^{-1}(j)}}^n m_{k\tau(k)}.
    \]

    Now let $\tau' = (i \ s) \tau$. Then $\tau'$ satisfies: $\tau'(i) = 1$, $\tau'(s) = \tau(i) \neq j$ and $\tau'(k) = k$ for all other $k \neq i,s$. Since $\tau'(s) \neq j$, $\tau'$ has a corresponding term in $m_{s1}D_{sj}$, which is
    \[
    \sgn(\tau') m_{s1} m_{i1} \prod_{\substack{k =1 \\ k \neq i, s, (\tau')^{-1}(j)}}^n m_{k\tau(k)}.
    \]
    Since $\tau^{-1}(j) = t$ for some $t \neq i,s$, by the above observations we have that $(\tau')^{-1}(j) = t$ as well. We also have $\sgn(\tau') = -\sgn(\tau)$. So this term is equal to
    \[
    -\sgn(\tau) m_{i1}m_{s1} \prod_{\substack{k=1 \\ k \neq i,s,\tau^{-1}(j)}}^n 
    m_{k\tau(k)}.
    \]

    Finally, $(s \ i) \tau' = \tau$. Hence this establishes a sign-reversing involution between the terms of the polynomials $m_{i1} D_{ij}$ that are not terms of $D_j$. So these terms cancel in the expression $\sum_{i=1}^n m_{i1} D_{ij}$ and this sum is equal to $D_j$.  
\end{proof}

We can now prove our main theorem regarding the generators of the extremal conditional independence ideal.

\begin{thm}\label{thm:CMGenerators}
    Let $A, B, C \subset [d]$ be disjoint. Then the ideal generated by all $(\#C + 2) \times (\#C +2)$ minors of $\CM(\Gamma_{A \cup C, B \cup C})$ is equal to
    \begin{equation}\label{eqn:CMGenerators}
    \langle \det \CM(\Gamma_{A', B'}) \mid A' \subset A \cup C, \ B' \subset B \cup C, \ \#A' = \#B' = \#C + 1 \rangle.
    \end{equation}
    Thus the extremal conditional independence ideal equals
    \[
    \eci_G = \sum_{ A \perp_e B \mid C} \langle \det \CM(\Gamma_{A', B'}) \mid A' \subset A \cup C, \ B' \subset B \cup C, \ \#A' = \#B' = \#C + 1 \rangle.
    \]
\end{thm}

\begin{proof}
To simplify the following proof, we introduce a variant of the Cayley-Menger matrix,
\begin{align}
    \widetilde{\CM}(\Gamma) := \CM(-2\Gamma) = \begin{pmatrix}
    \Gamma & \ones \\
    \ones^\top & 0
\end{pmatrix}.\label{eq:tildeCM}
\end{align}
Note that by scaling each row of $\CM(\Gamma)$ by $-2$ and the final column by $-1/2$, we have that
\[
\det(\widetilde{\CM}(\Gamma))  = (-2)^{d-1} \det(\CM(\Gamma)).
\]
Thus the ideal of all $(\#C+2) \times (\#C+2)$ minors of $\CM(\Gamma_{A\cup C, B \cup C})$ is equal to that of $\widetilde{\CM}(\Gamma_{A\cup C, B \cup C})$ and the ideal in Equation \eqref{eqn:CMGenerators} is the same if each $\CM(\Gamma_{A',B'})$ is replaced with $\widetilde{\CM}(\Gamma_{A',B'})$.

Each of the generators in Equation \eqref{eqn:CMGenerators} is a scalar multiple of a $(\#C + 2) \times (\#C +2)$ minor of $\CM(\Gamma_{A \cup C, B \cup C})$. So it suffices to show that all other such minors can be written as an ideal combination of the polynomials in Equation \eqref{eqn:CMGenerators}.

    First consider a matrix of variables $M = (m_{ij})_{i,j=1}^n$ and let $M_{\backslash i,\backslash j}$ denote $M$ with its $i$th row and $j$th column removed. Let $D_j$ and $D_{ij}$ be as in Lemmas \ref{lem:Dj} and \ref{lem:Dij}. Note that by moving the $i$th row to become the last row and the $j$th column to become the last column, we have that 
    \[D_{ij} = (-1)^{n-i}(-1)^{n-j} \det(\widetilde{\CM}(M_{\backslash i, \backslash j})) = (-1)^{i+j}\det(\widetilde{\CM}(M_{\backslash i, \backslash j})).\]
    Combining this fact with Lemmas \ref{lem:Dj} and \ref{lem:Dij}, we see that
    \begin{equation}\label{eq:CMIdealContainment}
    \det(M) \in \langle D_j \mid j \in [n] \rangle \subset \langle D_{ij} \mid i,j \in [n] \rangle = \langle \det(\widetilde{\CM}(M_{\backslash i, \backslash j})) \mid i,j \in [n] \rangle.
    \end{equation}
    
    Now we specialize to the case where $M$ is a $(\#C + 2) \times (\#C +2)$ minor of $\widetilde{\CM}(\Gamma_{A \cup C, B \cup C})$.
    Let $J$ denote the ideal in Equation \eqref{eqn:CMGenerators}. Then in particular, $\det(\widetilde{\CM}(M_{\backslash i, \backslash j})) \in J$ for all $i,j$. This fact and Equation \eqref{eq:CMIdealContainment} together imply that $\det(M) \in J$, as needed.

    The second statement follows immediately from the definition of the extremal conditional independence ideal.
\end{proof}

The minors corresponding to saturated statements $i \perp_e j \mid [d] \setminus \{i,j\}$ for each nonedge $(i,j)$ are all included in $\eci_G$. Therefore if $\Gamma$ is a matrix in the intersection of the variety $\mathcal{V}(\eci_G)$ with the strictly conditionally negative definite cone, $\mathcal{C}^d$, then $\CM(\Gamma)^{-1}$ has $ij$ entry equal to $0$ if $(i,j)$ is a nonedge of $G$. So $\Gamma$ is a variogram matrix in the model specified by $G$. In other words, the set of variogram matrices in this model is equal to $\mathcal{V}(\eci_G)$ intersected with $\mathcal{C}^d$.

However, the same may not be true on the level of vanishing ideals since $\eci_G$ is not necessarily prime. It may have associated primes whose varieties do not intersect the open cone $\mathcal{C}^d$ and thus are not statistically relevant. 
We now describe how to obtain the vanishing ideal of the model from $\eci_G$.

\begin{prop}
    Let $G$ be a connected graph. Then
    \[
    I_G = \sqrt{\eci_G :(\det\CM(\Gamma))^\infty}.
    \]
\end{prop}

\begin{proof}
    Consider an irreducible decomposition of the variety $\mathcal{V}(\eci_G) = W_1 \cup \dots \cup W_m$. Since $I_G$ is prime and $\eci_G \subset I_G$, we have that $V(I_G)$ is contained in one of these irreducible components. Since $V(I_G)$ is the Zariski closure of $\M_G$, denoted $\overline{\M_G}$, we may assume without loss of generality that $\overline{\M_G} \subset W_1$. 
    
    First we claim that for all $l >1$, $W_l \subset \mathcal{V}(\det\CM(\Gamma)).$ Indeed, let $X \in \mathcal{V}(\eci_G)$ be a matrix such that $\CM(X)$ is invertible. Then consider $\CM(X)^{-1}$. Note that if $(i,j)$ is not an edge in $G$, then the extremal conditional independence statement $\{i\} \perp_e \{j\} \mid [d]\setminus \{i,j\}$ holds in the model. So the minor of $\CM(X)$ obtained by deleting the $i$ row and $j$ column belongs to $\eci_G$. Hence this minor of $X$ vanishes and the $ij$ entry of $\CM(X)^{-1}$ is $0$. So $X$ arises as the variogram matrix corresponding to a weighted signed Laplacian of $G$ and therefore belongs to $\overline{\M_G}$. So $X$ is in $W_1$ and all other irreducible components of $\mathcal{V}(\eci_G)$ are contained in the vanishing locus of $\det(\CM(\Gamma))$. 
    
    This argument also shows that $W_1 = \overline{\M_G}$. Indeed, it shows that every variogram in $W_1$ is obtained from a weighted Laplacian of $G$. Hence the variety of the saturation $\eci_G:(\det\CM(\Gamma))^\infty$ is equal to $W_1$ by \cite[Chapter~4, Section~4, Theorem~10(iii)]{CLO}. By Hilbert's Nullstellensatz, we conclude that
    \[
    \mathcal{I}(\M_G) = I_G = \sqrt{\eci_G:(\det\CM(\Gamma))^\infty}.\qedhere
    \]
\end{proof}

\begin{ex}
It is not easy to directly find an example where $I_G \neq \eci_G$ since equality is satisfied for the graphs where standard Gr\"obner basis methods terminate quickly. However, we can adapt an example from Gaussian graphical models to obtain an extremal one where this is the case. Consider the directed graph with two parent nodes $\{6,7\}$ pointing to all nodes $\{1,2,3,4,5\}$. The CI statements coming from d-separation are the same as those of its \emph{moralization}, namely the undirected graph corresponding to a complete bipartite graph with an extra edge between the parents $6$ and $7$: $H:= K_{5,2} \cup \{(6,7)\}$. For this model, the \emph{pentad}, $f_\text{pentad}$, which can be found in \cite[Example~14.2.6]{sullivant2018}, belongs to the vanishing ideal of the undirected Gaussian graphical model given by $H$ but it is not in the Gaussian conditional independence ideal $\mathrm{CI}_H$.

We now consider $G:=H * \{8\}$ to be the suspension graph where the node $8$ is connected to every other node. Consider a weighted Laplacian $\Theta$ in the \HR graphical model specified by $G$. Deleting the $8$th row and column yields a $7 \times 7$ concentration matrix in the Gaussian graphical model specified by $H$ since each diagonal entry is independent of all other entries. Thus $\Sigma^{(8)}$ is a covariance matrix in this Gaussian model and the pentad, $f_\text{pentad}$, vanishes on $\Sigma^{(8)}$. Indeed, the ideal of the \HR model $I_G$ is equal to the vanishing ideal of the Gaussian model specified by $H$ transformed according to the covariance mapping in Equation \eqref{eq:Sigma^kFromGamma}.   If we consider the transformed pentad, 
\begin{align*}
  \tilde{f} =   &-(-g_{23} + g_{28} + g_{38})  (-g_{14} + g_{18} + g_{48})  (-g_{34} + g_{38} + g_{48})  (-g_{15} + g_{18} + g_{58})  (-g_{25} + g_{28} + g_{58}) \\
&+ (-g_{13} + g_{18} + g_{38})  (-g_{24} + g_{28} + g_{48})  (-g_{34} + g_{38} + g_{48})  (-g_{15} + g_{18} + g_{58})  (-g_{25} + g_{28} + g_{58}) \\
&+ (-g_{23} + g_{28} + g_{38})  (-g_{14} + g_{18} + g_{48})  (-g_{24} + g_{28} + g_{48})  (-g_{15} + g_{18} + g_{58})  (-g_{35} + g_{38} + g_{58}) \\
&- (-g_{12} + g_{18} + g_{28})  (-g_{24} + g_{28} + g_{48})  (-g_{34} + g_{38} + g_{48})  (-g_{15} + g_{18} + g_{58})  (-g_{35} + g_{38} + g_{58}) \\
&- (-g_{13} + g_{18} + g_{38})  (-g_{14} + g_{18} + g_{48})  (-g_{24} + g_{28} + g_{48})  (-g_{25} + g_{28} + g_{58})  (-g_{35} + g_{38} + g_{58}) \\
&+ (-g_{12} + g_{18} + g_{28})  (-g_{14} + g_{18} + g_{48})  (-g_{34} + g_{38} + g_{48})  (-g_{25} + g_{28} + g_{58})  (-g_{35} + g_{38} + g_{58})\\
&- (-g_{13} + g_{18} + g_{38})  (-g_{23} + g_{28} + g_{38})  (-g_{24} + g_{28} + g_{48})  (-g_{15} + g_{18} + g_{58})  (-g_{45} + g_{48} + g_{58})\\
&+ (-g_{12} + g_{18} + g_{28})  (-g_{23} + g_{28} + g_{38})  (-g_{34} + g_{38} + g_{48})  (-g_{15} + g_{18} + g_{58})  (-g_{45} + g_{48} + g_{58})\\
&+ (-g_{13} + g_{18} + g_{38})  (-g_{23} + g_{28} + g_{38})  (-g_{14} + g_{18} + g_{48})  (-g_{25} + g_{28} + g_{58})  (-g_{45} + g_{48} + g_{58})\\
&- (-g_{12} + g_{18} + g_{28})  (-g_{13} + g_{18} + g_{38})  (-g_{34} + g_{38} + g_{48})  (-g_{25} + g_{28} + g_{58})  (-g_{45} + g_{48} + g_{58})\\
&- (-g_{12} + g_{18} + g_{28})  (-g_{23} + g_{28} + g_{38})  (-g_{14} + g_{18} + g_{48})  (-g_{35} + g_{38} + g_{58})  (-g_{45} + g_{48} + g_{58})\\
&+ (-g_{12} + g_{18} + g_{28})  (-g_{13} + g_{18} + g_{38})  (-g_{24} + g_{28} + g_{48})  (-g_{35} + g_{38} + g_{58})  (-g_{45} + g_{48} + g_{58})
\end{align*}
we see that $\tilde{f} \in I_G \setminus \eci_G$. Indeed, the vertices $6$ and $7$ do not appear in any subscript of a variable in $\tilde{f}$. But these vertices are in every separation statement in $G$, so every term of every polynomial in $\eci_G$ must have a variable with $6$ or $7$ as an index. So $\tilde{f}$ does not belong to $\eci_G$.
\end{ex}

\section{Matrix Completion and the Extremal Maximum Likelihood Degree}\label{sec:MLdegree}

Let $G=(V,E)$ be an undirected graph. We call a symmetric matrix $\mathring{\Gamma}$ a \emph{partially-specified conditionally negative definite variogram} with respect to $G$, if its diagonal is zero, and all principal submatrices corresponding to cliques of $G$ are conditionally negative definite. We call the matrix partially-specified because in the following discussion, we only consider the entries of $\mathring{\Gamma}$ corresponding to entries of $G$ and we ``forget'' those corresponding to non-edges. The ``forgotten" entries are represented by a placeholder ``?". For $\mathring{\RR}=\RR\cup\{?\}$, we then write $\mathring{\Gamma}\in \mathring{\RR}^{\#V\times \#V}$.
Given a partially-specified CND variogram, \citet[Definition 4.1]{HES2022}\label{def:MCP} define the following matrix completion problem.
\begin{defi}\label{def:matrixcompl}
Let $G=(V,E)$ be an undirected graph and let $\mathring{\Gamma}$ be a partially-specified conditionally negative definite matrix with respect to $G$. 
We call a conditionally negative matrix $\Gamma$, with corresponding signed Laplacian $\Theta$, that satisfies
\begin{align*}
    \Gamma_{ij}&=\mathring{\Gamma}_{ij}, \qquad\forall (i,j)\in E,\\
    \Theta_{ij}&=0, \;\;\; \qquad\forall (i,j)\not\in E,
\end{align*}
a solution to the matrix completion given by $G$ and $\mathring{\Gamma}$.
\end{defi}
\citet{HES2022} show that the matrix completion problem in Definition~\ref{def:matrixcompl} links to maximum likelihood estimation for degenerate Gaussian distributions where the covariance is constrained to be a signed Laplacian matrix.
In fact, for a connected undirected graph $G=(V,E)$ and some sample variogram matrix $\overline{\Gamma}$ that gives rise to a partially-specified CND variogram, the solution to the matrix completion problem is the variogram corresponding to the unique solution of
\begin{align}
    \max_{\Theta \in \mathbb{U}^d}\left\{ \log\Det(\Theta) +\frac{1}{2}\tr\left(\Theta \overline{\Gamma}\right)\right\},\qquad \text{s.t. } \Theta_{ij}=0\;\; \forall (i,j)\not\in E, \label{eq:LCGGM-MLE}
\end{align}
where $\Det$ denotes the pseudo-determinant.

The similarity of the degenerate Gaussian log-likelihood with the \HR exponent measure density $\lambda$ provides an attractive option for parameter estimation, as the H\"usler--Reiss log-likelihood is not easily tractable due to its normalization constant  \citep{HES2022,REZ2021}.
Here, the sample variogram is chosen as the empirical variogram (see \eqref{eq:emp_vario}), which is a consistent estimator of the variogram parameter of the \HR distribution.
A key observation is that $\frac{\partial}{\partial \Theta_{ij}}\left\{ \log\Det(\Theta) +\frac{1}{2}\tr\left(\Theta \overline{\Gamma}\right)\right\}=-\Gamma_{ij}$ for all $i\neq j$ \citep[Proposition~A.5]{REZ2021}. Therefore, the critical points of \eqref{eq:LCGGM-MLE} are given by the matrix completion problem with partially-specified conditionally negative definite matrix $\overline{\Gamma}$.
This means that the solution of the matrix completion problem
in Definition \ref{def:matrixcompl},
where now $\overline{\Gamma}$ is the empirical variogram of \citet{EV2020}, is a surrogate maximum likelihood estimator for the \HR graphical model with respect to $G$.

The matrix completion problem in Definition~\ref{def:matrixcompl} is a semialgebraic problem, since we are searching for solutions of a polynomial equation system under the constraint that the solution is conditionally negative definite.
Without the definiteness constraints, this is a purely algebraic problem, such that techniques from commutative algebra can help analyze the solution space of the problem.
Thus it is natural to make the following definition. 
\begin{defi}
  Let $\mathring{\Gamma}$ be a \emph{generic} partially-specified conditionally negative definite matrix with respect to a graph $G=(V,E)$. We call the number of complex solutions of the matrix completion problem in Definition~\ref{def:matrixcompl}, or equivalently the number of complex critical points of the optimization problem~\eqref{eq:LCGGM-MLE}, without the restriction of conditional negative definiteness, the \emph{extremal (surrogate) maximum likelihood degree}. We denote this as $\emld(G).$
\end{defi}

\begin{ex}\label{ex:complete}
    Let $G=K_d$ be the complete graph on $d$ vertices. Then the solution $\Gamma$ to the matrix completion problem in Definition \ref{def:matrixcompl} is completely specified by the conditionally negative definite variogram $\overline{\Gamma}$, as there are no constraints on the corresponding $\Theta$. Thus, $\emld(K_d)=1$ for any $d$.
\end{ex}

We find a multiplicative formula for the extremal maximum likelihood degrees of graphs that can be decomposed into two components. 
\begin{thm}\label{thm:multiplicative}
Let $G=(V,E)$ be a connected graph that is decomposed by sets $A,B$ and a separating clique $C$, such that $A\cup B\cup C=V$. Then 
\begin{equation}
    \emld(G) = \emld (G[A\cup C]) \cdot \emld(G[B\cup C]).
\end{equation}
\end{thm}
\begin{proof}
From \citet[Theorem~2]{ET2023} it follows that the \HR precision matrix of the marginals $AC$ and $BC$ are
\begin{align*}
    \theta(\Gamma_{A\cup C,A\cup C})&=\Theta_{AC,AC}-\Theta_{AC,B}\Theta_{B,B}^{-1}\Theta_{B,AC}, \\
    \theta(\Gamma_{B\cup C,B\cup C})&=\Theta_{BC,BC}-\Theta_{BC,A}\Theta_{A,A}^{-1}\Theta_{A,BC}.
\end{align*}
As $C$ is a separating set, we observe that $\Theta_{A,B}=\boldsymbol{0}$, and we can conclude that 
\begin{align}
    \theta(\Gamma_{A\cup C,A\cup C})=\begin{pmatrix}
        \Theta_{A,A}&\Theta_{A,C} \\
        \Theta_{C,A}& *\\
    \end{pmatrix}, \qquad \theta(\Gamma_{B\cup C,B\cup C})=\begin{pmatrix}
        *&\Theta_{C,B} \\
        \Theta_{B,C}& \Theta_{B,B}\\    \end{pmatrix}.\label{eq:Theta_sub_separated}
\end{align}

Assume a generic summary statistic $\mathring{\Gamma}$.
The equations for the completion problem of $G$ are
$\Gamma_{ij}=\mathring{\Gamma}_{ij}$ for all $(i,j)\in E$ and $\Theta_{ij}=0$ for all $(i,j)\not\in E$.
From the characterizations~\eqref{eq:Theta_sub_separated} we find that the completion problem of the marginal with variogram $\Gamma_{A\cup C,A\cup C}$ is thus
$\Gamma_{ij}=\mathring{\Gamma}_{ij}$ for all $(i,j)\in E[A\cup C]$ and $\Theta_{ij}=0$ for all $(i,j)\not\in E[A\cup C]$.
Let now $\widehat\Gamma$ denote the solution of the matrix completion problem with respect to $G$ for some generic summary statistic $\mathring{\Gamma}$.
Then $\widehat{\Gamma}_{A\cup C,A\cup C}$ satisfies the marginal completion problem with respect to $G[A\cup C]$, and $\widehat{\Gamma}_{B\cup C,B\cup C}$ satisfies the marginal completion problem with respect to $G[B\cup C]$. Clearly, the marginal completion solutions agree on the block $\Gamma_{C,C}$, as they both equal $\mathring{\Gamma}_{C,C}$ there.
Thus, given solutions for both marginal problems, we can obtain a solution for the original completion problem by solving the problem
\[\Gamma=\begin{pmatrix}
    \widehat{\Gamma}_{A,A}&\widehat{\Gamma}_{A,C}&?\\
    \widehat{\Gamma}_{C,A}&\widehat{\Gamma}_{C,C}&\widehat{\Gamma}_{C,B}\\
    ?&\widehat{\Gamma}_{B,C}&\widehat{\Gamma}_{B,B}\\
\end{pmatrix}\qquad \Theta=\begin{pmatrix}
    *&*&\bs 0\\
    *&*&*\\
    \bs 0 &*&*\\
\end{pmatrix}.\]
This means that every pair of solutions to the marginal completion problems gives rise to a solution of the original completion problem, in the sense that $\Theta$ is completely specified by the marginal solutions in \eqref{eq:Theta_sub_separated}. As the function $\Gamma=\gamma(\Theta^+)$ is a bijection, the corresponding variogram $\Gamma$ is also totally specified. Hence, the total number of solutions can be expressed as a product.
\end{proof}

As a corollary of Theorem~\ref{thm:multiplicative}, we conclude that chordal graphs have extremal ML degree one.
\begin{cor}\label{cor:one}
    Let $G$ be a chordal graph. Then $\emld(G)=1$.
\end{cor}
\begin{proof}
    Every chordal graph can be obtained successively by gluing cliques. Each clique is a complete graph and has extremal ML degree one by Example \ref{ex:complete}. By applying Theorem~\ref{thm:multiplicative} and induction on the number of cliques, the result follows. 
\end{proof}

In the following, we will derive a closed-form solution for the conditionally negative definite matrix completion problem for chordal graphs.
For some $d\times d$-matrix $M$ and some index set $A\subset V$ with $\#V=d$, let $\big[M_{A,A}\big]^d$ be the $d\times d$-matrix that equals $M_{i,j}$ for $i,j\in A$, and zero elsewhere.
We remind the reader that for some CND variogram matrix $\Gamma$, the function $\theta(\Gamma)$ returns the corresponding signed Laplacian matrix. Furthermore, note that $\theta([\Gamma_{ii}])=[0]$ for all $i\in V$.
First, we study the setting where a connected graph is composed of only two cliques.
\begin{lemma}\label{lem:MLEformula}
    Let $G=(V,E)$ be a connected graph with $\#V=d$ that is composed of two cliques $A,B$ with $A\cup B=V$.
    Let $\mathring{\Gamma}$ be a partially-specified conditionally negative definite variogram with respect to $G$.
    Then, the solution to the matrix completion problem is given by
    \[\Theta=\big[\theta(\mathring{\Gamma}_{A,A})\big]^d+\big[\theta(\mathring{\Gamma}_{B,B})\big]^d-\big[\theta(\mathring{\Gamma}_{A\cap B, A\cap B})\big]^d\]
\end{lemma}
\begin{proof}
Let $k\in A\cap B$, and let $G_k=(V\setminus\{k\},E_k)$ be the subgraph of $G$ where the node $k$ was removed. Define the (bijective) map $\varphi_k(\Gamma)=\frac{1}{2}(\Gamma_{ik}+\Gamma_{jk}-\Gamma_{ij})_{i,j\in V\setminus\{k\}}$, and let $S^{(k)}=\varphi_k(\mathring{\Gamma})$.
The positive definite matrix completion problem
\begin{align*}
    \Sigma^{(k)}_{ij}&=S^{(k)}_{ij}, \qquad\forall (i,j)\in E_k,\\
    \Theta^{(k)}_{ij}&=0, \;\;\;\qquad\forall (i,j)\not\in E_k,
\end{align*}
has a unique solution $\Sigma^{(k)}$ if $S^{(k)}_{A,A}=\varphi_k(\mathring{\Gamma}_{A,A})$ and $S^{(k)}_{B,B}=\varphi_k(\mathring{\Gamma}_{B,B})$ are positive definite.
This condition is satisfied because $\varphi_k$ maps strictly CND matrices to positive definite matrices \citep{HES2022}.
Let $K_{[D]}:=(S^{(k)}_{D,D})^{-1}$ for any $D\subset V$ except for $D=\{k\}$, where we set $K_{[k]}=0$.
By \citet[Proposition~5.6]{Lauritzen96}, the solution satisfies
\[\Theta^{(k)}=(\Sigma^{(k)})^{-1}=\big[K_{[A]}\big]^{d-1}+\big[K_{[B]}\big]^{d-1}-\big[K_{[A\cap B]}\big]^{d-1}.\]
By \citet[Lemma 4.3]{HES2022}, $\Gamma = \varphi_k^{-1}(\Sigma^{(k)})$ is the solution to the conditionally negative matrix completion problem.
Furthermore, observe that for $k\in D\subseteq V$, it is $\theta^{(k)}(\mathring{\Gamma}_{D,D})=K_{[D]}$. 
Thus,
\[\Theta^{(k)}=\big[\theta^{(k)}(\mathring{\Gamma}_{A,A})\big]^{d-1}+\big[\theta^{(k)}(\mathring{\Gamma}_{B,B})\big]^{d-1}-\big[\theta^{(k)}(\mathring{\Gamma}_{A\cap B,A\cap B})\big]^{d-1},\]
Now, as 
\[\Theta_{ij}=\begin{cases}
    \Theta^{(k)}_{ij}, &i\neq k ,j\neq k,\\
    -\sum_{u\neq k}\Theta^{(k)}_{iu}, & i\neq k, j=k,\\
    \sum_{u\neq k,v\neq k}\Theta^{(k)}_{uv}, & i= k, j=k,\\
\end{cases}\]
is a linear map, the result follows.
\end{proof}

This lemma gives rise to the following theorem.
\begin{thm}\label{thm:chordalmld1}
Let $G=(V,E)$ be a chordal connected graph with $\#V=d$,  set of maximal cliques $\mathcal{C}$ and set of separating cliques $\mathcal{S}$. Let $\mathring{\Gamma}$ be a partially-specified conditionally negative definite variogram with respect to $G$.
Then, the corresponding matrix completion problem is solved by
\[\Theta = \sum_{C\in \mathcal{C}}\big[\theta(\mathring{\Gamma}_{C,C})\big]^d-\sum_{S\in \mathcal{S}}\nu(S)\big[\theta(\mathring{\Gamma}_{S,S})\big]^d, \]
where $\nu(S)$ is the number of times the separating clique $S$ appears in a perfect sequence.
\end{thm}
 \begin{proof}
 By \citet[Proposition~2.17]{Lauritzen96}, the cliques of any chordal graph satisfy the running intersection property, that is, the cliques allow an ordering $C_1,\ldots,C_m$ with $m=\#\mathcal{C}$ such that there exists for every $j=2,\ldots m$ an $i < j$ with
\[C_j \cap \bigcup_{\ell=1}^{j-1} C_\ell= C_j\cap C_i. \]
This means that the intersection with all previously seen cliques in the sequence lies in only one clique.
The nonempty intersections form the set of separating cliques $\mathcal{S}$.
Following \citet{HES2022}, let $\Gamma_1=\mathring{\Gamma}_{C_1,C_1}$, $V_n=C_1\cup \ldots\cup C_n$ and $K_n=V_n\times V_n$. 
For $n=2,\ldots,m$, let $\mathring{\Gamma}_n \in \mathring{\RR}^{\#V_n\times \#V_n}$ with
zeros on the diagonal and otherwise given by 
\[(\mathring{\Gamma}_n)_{ij}=\begin{cases}
    (\Gamma_{n-1})_{ij}, & (i,j)\in K_{n-1},\\
    \mathring{\Gamma}_{ij}, & (i,j) \in E|_{V_n} \setminus K_{n-1},\\ 
    ?,& \text{otherwise.}
\end{cases}\]
Let $\Gamma_n$ denote the solution to the matrix completion problem with partially-specified CND matrix $\mathring{\Gamma}_n$ and graph $G_n=(V_n, K_{n-1}\cup E_{|V_n} )$.
By Lemma~\ref{lem:MLEformula}, the H\"usler--Reiss precision matrix corresponding to $\Gamma_n$ is
\begin{align*}
    \theta(\Gamma_n)&=\big[\theta((\mathring{\Gamma}_{n})_{V_{n-1},V_{n-1}})\big]^{\# V_n}+\big[\theta((\mathring{\Gamma}_{n})_{C_n,C_n})\big]^{\#V_n}-\big[\theta((\mathring{\Gamma}_{n})_{V_{n-1}\cap C_n,V_{n-1}\cap C_n})\big]^{\#V_n}\\
    &=\big[\theta(\Gamma_{n-1})\big]^{\#V_n}+\big[\theta(\mathring{\Gamma}_{C_n,C_n})\big]^{\#V_n}-\big[\theta(\mathring{\Gamma}_{V_n\cap C_n, V_n\cap C_n})\big]^{\#V_n}.
\end{align*}
Now, by \citet[Proposition~4.4]{HES2022}, $\Gamma_m$ exists and is the unique solution of the completion problem with partially-specified conditionally negative definite matrix $\mathring{\Gamma}$ and graph $G$. 
The theorem follows by recursively applying this solution.
 \end{proof}

\begin{figure}
\begin{tikzpicture}[node distance=20mm]

\node[node] (2) at (0,2) {2};
\node[node] (5) at (4,2) {5};

\node[node] (1) at (-2,0) {1};
\node[node] (4) at (2,0) {4};

\node[node] (3) at (0,-2) {3};
\node[node] (6) at (4,-2) {6};

\draw[edge] (1) -- (2);
\draw[edge] (1) -- (3);
\draw[edge] (4) -- (5);
\draw[edge] (4) -- (6);
\draw[edge] (1) -- (4);
\draw[edge] (2) -- (4);
\draw[edge] (3) -- (4);

\end{tikzpicture}\caption{A decomposable graph}\label{fig:fish}
\end{figure}
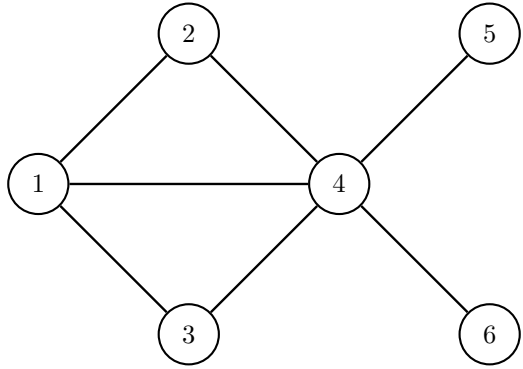
We provide the following example to illustrate Theorem~\ref{thm:chordalmld1}.
\begin{ex}
    Let $G$ be the graph in Figure~\ref{fig:fish} and assume that we are given a partially-specified CND variogram matrix $\mathring{\Gamma}$ with respect to $G$.
The parameter matrix of the graphical model is
\[
\Theta =
\begin{pmatrix}
\Theta_{11} & \Theta_{12} & \Theta_{13} & \Theta_{14} & 0 & 0 \\
\Theta_{21} & \Theta_{22} & 0 & \Theta_{24} & 0 & 0 \\
\Theta_{31} & 0 & \Theta_{33} & \Theta_{34} & 0 & 0 \\
\Theta_{41} & \Theta_{42} & \Theta_{43} & \Theta_{44} & \Theta_{45} & \Theta_{46} \\
0 & 0 & 0 & \Theta_{54} & \Theta_{55} & 0 \\
0 & 0 & 0 & \Theta_{64} & 0 & \Theta_{66}
\end{pmatrix},
\]
with zero entries specified by missing edges in $G$.
The graph $G$ has cliques
\[C_1=\{1,2,4\},\qquad C_2=\{1,3,4\},\qquad C_3=\{4,5\},\qquad C_4=\{4,6\}, \]
and separating cliques
\[S_1=\{1,4\},\qquad S_2=\{4\}\]
with multiplicities
\[\nu(S_1)=1,\qquad \nu(S_2)=2.\]
Using Theorem~\ref{thm:chordalmld1}, we find the following formula for the matrix completion problem.
For convenience we write $\mathring{\Gamma}_{C}$
instead of $\mathring{\Gamma}_{C,C}$.
\begin{align*}
    \widehat{\Theta}&=
\begin{pmatrix}
\theta_{11}(\mathring{\Gamma}_{C_1}) & \theta_{12}(\mathring{\Gamma}_{C_1}) & 0 & \theta_{14}(\mathring{\Gamma}_{C_1}) & 0 & 0 \\
\theta_{12}(\mathring{\Gamma}_{C_1}) & \theta_{22}(\mathring{\Gamma}_{C_1}) & 0 & \theta_{24}(\mathring{\Gamma}_{C_1}) & 0 & 0 \\
0 & 0 & 0 & 0 & 0 & 0 \\
\theta_{14}(\mathring{\Gamma}_{C_1}) & \theta_{24}(\mathring{\Gamma}_{C_1}) & 0 & \theta_{44}(\mathring{\Gamma}_{C_1}) & 0 & 0 \\
0 & 0 & 0 & 0 & 0 & 0 \\
0 & 0 & 0 & 0 & 0 & 0
\end{pmatrix}
\;+\;
\begin{pmatrix}
\theta_{11}(\mathring{\Gamma}_{C_2}) & 0 & \theta_{13}(\mathring{\Gamma}_{C_2}) & \theta_{14}(\mathring{\Gamma}_{C_2}) & 0 & 0 \\
0 & 0 & 0 & 0 & 0 & 0 \\
\theta_{13}(\mathring{\Gamma}_{C_2}) & 0 & \theta_{33}(\mathring{\Gamma}_{C_2}) & \theta_{34}(\mathring{\Gamma}_{C_2}) & 0 & 0 \\
\theta_{14}(\mathring{\Gamma}_{C_2}) & 0 & \theta_{34}(\mathring{\Gamma}_{C_2}) & \theta_{44}(\mathring{\Gamma}_{C_2}) & 0 & 0 \\
0 & 0 & 0 & 0 & 0 & 0 \\
0 & 0 & 0 & 0 & 0 & 0
\end{pmatrix}\\
&\;+\; \begin{pmatrix}
0 & 0 & 0 & 0 & 0 & 0 \\
0 & 0 & 0 & 0 & 0 & 0 \\
0 & 0 & 0 & 0 & 0 & 0 \\
0 & 0 & 0 & \theta_{44}(\mathring{\Gamma}_{C_3}) & \theta_{45}(\mathring{\Gamma}_{C_3}) & 0 \\
0 & 0 & 0 & \theta_{45}(\mathring{\Gamma}_{C_3}) & \theta_{55}(\mathring{\Gamma}_{C_3}) & 0 \\
0 & 0 & 0 & 0 & 0 & 0
\end{pmatrix}
\;+\; \begin{pmatrix}
0 & 0 & 0 & 0 & 0 & 0 \\
0 & 0 & 0 & 0 & 0 & 0 \\
0 & 0 & 0 & 0 & 0 & 0 \\
0 & 0 & 0 & \theta_{44}(\mathring{\Gamma}_{C_4}) & 0 & \theta_{46}(\mathring{\Gamma}_{C_4}) \\
0 & 0 & 0 & 0 & 0 & 0 \\
0 & 0 & 0 & \theta_{46}(\mathring{\Gamma}_{C_4}) & 0 & \theta_{66}(\mathring{\Gamma}_{C_4})
\end{pmatrix}\\
&\;-1\;\begin{pmatrix}
\theta_{11}(\mathring{\Gamma}_{S_1}) & 0 & 0 & \theta_{14}(\mathring{\Gamma}_{S_1}) & 0 & 0 \\
0 & 0 & 0 & 0 & 0 & 0 \\
0 & 0 & 0 & 0 & 0 & 0 \\
\theta_{14}(\mathring{\Gamma}_{S_1}) & 0 & 0 & \theta_{44}(\mathring{\Gamma}_{S_1}) & 0 & 0 \\
0 & 0 & 0 & 0 & 0 & 0 \\
0 & 0 & 0 & 0 & 0 & 0
\end{pmatrix} \;-2\;\begin{pmatrix}
0 & 0 & 0 & 0 & 0 & 0 \\
0 & 0 & 0 & 0 & 0 & 0 \\
0 & 0 & 0 & 0 & 0 & 0 \\
0 & 0 & 0 & 0 & 0 & 0 \\
0 & 0 & 0 & 0 & 0 & 0 \\
0 & 0 & 0 & 0 & 0 & 0
\end{pmatrix}.
\end{align*}
\end{ex}

\begin{prop}\label{prop:bipartite}
The extremal ML degree of the complete bipartite graph $K_{2,n}$ is $2n$.
\end{prop}

The proof follows the argument used in the proof of \citet[Theorem 4.3]{GeometryML12}, where the Gaussian ML degree of $K_{2,n}$ is shown to be $2n+1$.
The difference in the number of solutions is due to the fact that
the matrix $\CM(\Gamma)$ used in H\"usler--Reiss models
is augmented by a row and column of $1$-vectors and a $0$ on the new diagonal position.
Thus the degree of the minors of interest are reduced by one from the Gaussian case.

\begin{proof}
Instead of working with the standard form of the matrix $\CM (\Gamma)$,
again we use the form $\widetilde{\CM}(\Gamma)$ as in \eqref{eq:tildeCM}.
The analogous formula to \eqref{eq:FiedlerBapat} is given by \citet[Equation 11]{FJA2023}, namely
\begin{equation*}
   \widetilde{\CM} (\Gamma)\;=\; -2 \begin{pmatrix}
	\Theta&* \\
     *& *\\
\end{pmatrix}^{-1},
\end{equation*}
where stars are rational in the entries of $\Gamma$.

We write the matrix $\widetilde{\CM} (\Gamma)$ in block form, setting $A=\{1,\dots, n\}$
for the entries corresponding to the first $n$ nodes of the bipartite graph (without any edge between them),
and $B = \{n+1, n+2\}$ for the remaining two variables, plus the last row/column in the matrix.
We use $g_{ij}$ for the (known) entries of $\widetilde{\CM} (\Gamma)$ corresponding to edges $(i,j)$ and fill the diagonal with zeros,
while we use $y$ for the entry $(n+1,n+2)$ and $z_{ij}$ for all remaining nonedges.
Thus $\widetilde{\CM} (\Gamma)$ has the following block form
\begin{equation*}
\widetilde{\CM} (\Gamma) = 
\left(\begin{array}{cccc|ccc}
0 & z_{12} & \cdots & z_{1n} & g_{1n+1} & g_{1n+2} & 1\\
z_{12} & 0 & \cdots & z_{2n} & g_{2n+1} & g_{2n+2} & 1\\
\vdots & \vdots & \ddots & \vdots & \vdots & \vdots & \vdots\\
z_{13} & z_{23} & \cdots & 0 & g_{34} & g_{35} & 1\\
\hline
g_{14} & g_{24} & \cdots& g_{3n} & 0 & y & 1\\
g_{15} & g_{25} & \cdots& g_{3n} & y & 0 & 1\\
1 & 1 & \cdots & 1 & 1 & 1 & 0\\
\end{array}\right).
\end{equation*}
The goal is to find the number of matrix completions in the unknowns $y$ and $z_{ij}$.
The Schur complement
\[
 \Gamma_{A,A} -  \begin{pmatrix}
     \Gamma_{A,B}&\ones
 \end{pmatrix}\widetilde{\CM} (\Gamma_{B,B})^{-1}\begin{pmatrix}
     \Gamma_{B,A}\\
     \ones ^\top
 \end{pmatrix}
\]
has the same zeros as the $(A,A)$ part of the  inverse matrix $\widetilde{\CM} (\Gamma)^{-1}$.
Since $ \Theta_{AA}$ is a diagonal matrix,
all nondiagonal entries of the Schur complement above are zero.
This gives rise to equations
\begin{equation}
\label{eq:zterms}
    z_{ij} =  \tfrac{1}{2}(-y + g_{i,n+1} + g_{i,n+2} + g_{j,n+1} +g_{j,n+2}- \frac{g_{i,n+1}g_{j,n+1}-g_{i,n+1}g_{i,n+2}-g_{i,n+2}g_{j,n+1}+g_{i,n+2}g_{j,n+2}}{y})
\end{equation}
for all nonedges $(i,j)$ except $(n+1,n+2)$.
Furthermore, $\Theta_{n+1,n+2} = 0$,
therefore the minor $\det(\widetilde{\CM} (\Gamma_{\backslash (n+1),\backslash (n+2)}) )$ vanishes,
giving rise to a homogeneous polynomial in $z_{ij}$ of degree $n$.
Multiplying this polynomial by $y^n$
and getting rid of $z_{ij}$ by means of Equations \eqref{eq:zterms}
(substituting terms $yz_{ij}$ with degree $2$ polynomials in $y$ only)
gives an equation in $y$ of degree $2n$.
The remaining terms $z_{ij}$ are then linearly recoverable.
\end{proof}

\begin{ex}\label{ex:4cycle}
    Let $G=C_4$ be the undirected 4-cycle. This is the smallest non-chordal graph, and can be interpreted as the complete bipartite graph $K_{2,2}$ on partite sets $\{1,3\}$ and $\{2,4\}$. Thus, as a special case of Proposition \ref{prop:bipartite} when $n=2$, we obtain that the extremal ML degree of the $4$-cycle is now $2\cdot2 = 4$, as opposed to the Gaussian ML degree of $5$. This means that the extremal ML estimation problem for the 4-cycle is solvable by radicals whereas the Gaussian ML estimation problem is not.
\end{ex}

The matrix completion problem for variogram matrices (see Definition~\ref{def:matrixcompl}) can always be rephrased in terms of polynomial equations and inequalities in $\Sigma^{(k)}$ and $\Theta^{(k)}$ for any $k\in V$. We additionally impose linear constraints on the precision matrix in the form of vanishing row sums. However, if $k$ is connected to all other nodes, this is equivalent to a positive definite completion problem with respect to $G\setminus k$.
 Indeed, in the case of suspension graphs, we can relate the extremal ML-degree to the ML-degree of a Gaussian problem in the following way. We denote by $\mathrm{MLD}(G)$ the maximum likelihood degree of the Gaussian graphical model associated to $G$; this is the degree of the corresponding \emph{positive definite} matrix completion problem, as described in \citet{sturmfels2010multivariate}.

\begin{prop}
For any graph $G$, the extremal ML degree of its suspension equals its Gaussian ML degree:
    $$\emld(G \cup \{0\}) = \mld(G).$$
\end{prop}
\begin{proof}
    In the graph $G\cup\{0\}=(V_0,E_0)$, the node $0$ is connected to any other node. Thus, $\Gamma_{i0}=\overline{\Gamma}_{i0}$ for all $i\in V$, i.e.~there are no zero constraints imposed in the $0$-th row and column of $\Theta$. It follows that the matrix completion problem
    \begin{align*}
    \Gamma_{ij}&=\overline{\Gamma}_{ij}, \qquad\forall (i,j)\in E_0,\\
    \Theta_{ij}&=0, \qquad\forall (i,j)\not\in E_0,
\end{align*}
is equivalent to the completion problem
\begin{align*}
    \Sigma_{ij}^{(0)}&=S_{ij}^{(0)}, \qquad\forall (i,j)\in E,\\
    \Theta_{ij}^{(0)}&=0, \qquad\forall (i,j)\not\in E,
\end{align*}
where $\Sigma^{(0)}=\varphi_0(\Gamma)$ and $S^{(0)}=\varphi_0(\overline{\Gamma})$. 
The latter is the positive definite completion problem with respect to $G$.
Hence, the \emld{} of the suspension graph $G\cup \{0\}$ is equal to the \mld{} of $G$.
\end{proof}

The examples above suggest the following conjecture.
\begin{conj}\label{conj:degree}
    For any graph $G$, the extremal ML degree is at most its Gaussian ML degree: 
    \begin{equation}\label{eq:conj-ineq}
        \emld(G) \leq \mld(G).
    \end{equation}
\end{conj}

Note that in \eqref{eq:conj-ineq}, equality holds for chordal graphs by Corollary \ref{cor:one} (both sides are 1), while the inequality can also be strict as is the case for $K_{2,n}$ from Proposition \ref{prop:bipartite} (both sides differ by 1). Moreover, we observe that the difference between the two ML degrees can grow exponentially with $n$:

\begin{prop}\label{prop:mldrelations}
    The following results hold for the Gaussian ML degree and the extremal ML degree of the $n$-cycle $C_n$, for all $n\geq 3$:
\begin{enumerate}
    \item[(i)] $\mld(C_n)-\emld(C_n) = 2^{n-2}(n-5)+n+1$
    \item[(ii)] $\mld(C_n) = \emld(C_n)^2 - \emld(C_{n-1})\cdot \emld(C_{n+1}).$
\end{enumerate}    
\end{prop}

\begin{proof}
    Both statements follow from straightforward computation by using the exact values for both ML degrees, namely:
    \begin{align}
        \mld(C_n)&= (n-3)2^{n-2}+1, \label{eq:mldcycle} \\
        \emld(C_n) &= 2^{n-1}-n. \label{eq:emldcycle}
    \end{align}
    Equation \eqref{eq:mldcycle} was conjectured by \citet{DSS2009} and proved recently by \citet{dinu2025proof}.
     Equation \eqref{eq:emldcycle} was proven in a different language in the context of homaloidal polynomials by \cite[Theorem 2.13]{cox2024homaloidal}.
\end{proof}

Proposition \ref{prop:mldrelations}(ii) suggests a deep connection between the two cyclic ML degrees, with some type of determinantal transformation between both sequences. We also note that our Conjecture \ref{conj:degree} would imply \cite[Conjecture 2.17]{cox2024homaloidal}, which is equivalent to the statement that a graph is chordal if and only if its extremal maximum likelihood degree is one. The fact that only chordal graphs have Gaussian ML degree one appears in \cite[Theorem 3]{sturmfels2010multivariate} and a full proof can be found in \citet{amendola2024maximum}. 

\section{Extremal Maximum Likelihood Thresholds} \label{sec:mlt}

In the Gaussian setting, the \emph{maximum likelihood threshold} of a graph is the number of data points that guarantees generic existence of the maximum likelihood estimator. In other words, it is the minimal $r$ such that the MLE exists for generic sample covariance matrices of rank $r$. In the setting of H\"usler-Reiss graphical models, the sample statistic for surrogate maximum likelihood estimation is the empirical variogram, 
see Equation~\eqref{eq:emp_vario}.
The empirical variogram is the average of sample variograms calculated for the subsamples on each positive canonical halfspace in the support of the \HR{} distribution.
But, as observations are typically included in many canonical half-spaces, there is no clear formula that determines the relationship between the dimensionality of the empirical variogram and the overall number of data points. However, for generic data points, the dimensionality of the empirical variogram does not decrease as the number of data points grows. That is, the dimensionality of the empirical variogram is at least weakly monotonic in the number of generic data points. Thus, we define the extremal maximum likelihood threshold as the minimal dimensionality of an empirical variogram matrix such that the surrogate MLE, as the solution of the matrix completion problem in Definition~\ref{def:matrixcompl}, exists with probability~1. 

The dimensionality of the variogram can be readily computed via the rank of the empirical covariance matrix. If one calculates the dimensionality of the empirical variogram matrix and finds that it is greater than or equal to the extremal maximum likelihood threshold, this provides a certificate the surrogate MLE exists almost surely. On the other hand, if the dimensionality is less than the extremal maximum likelihood threshold, then one may need more data points to ensure that the surrogate MLE exists. This is especially relevant in setting of multivariate extremes since extreme data points are harder to come by.

\begin{defi}
    The \emph{extremal maximum likelihood threshold} of $G$, denoted $\emlt(G)$, is the minimum $m$ such that a generic empirical variogram matrix of dimensionality $m$ admits a strictly CND completion. The \emph{weak} maximum likelihood threshold is the minimum $m$ such that there exists an empirical variogram matrix $\Gamma_0$ of dimensionality $m$ that admits a strictly CND completion. 
\end{defi}

The extremal maximum likelihood threshold can be interpreted as the minimum $m$ such that the surrogate \HR maximum likelihood estimate exists for a generic empirical variogram matrix of dimensionality $m$. Analogously, if the weak maximum likelihood threshold is less than or equal to $m$, then the surrogate MLE exists for a Euclidean open set of empirical variogram matrices of dimensionality $m$. 

Now we establish some bounds on the extremal maximum likelihood threshold of a graph. We require the following graph-theoretic definitions. For a graph $G$, the number of vertices in the largest complete subgraph of $G$ is called the \emph{clique number} of $G$ and is denoted $q(G)$. A \emph{chordal cover} of $G = (V,E)$ is a graph $G' = (V,E')$ on the same vertex set as $G$ such that $E \subset E'$ and $G'$ is chordal. Let $G'$ be the chordal cover of $G$ with the smallest clique number. Then the \emph{treewidth} of $G$ is $\tw(G) = q(G') -1$. These values bound the maximum likelihood threshold in the Gaussian setting as follows.

\begin{thm}[\citet{buhl1993existence}, Corollary~3.3]
\label{thm:ineqGaussian}
    In the Gaussian graphical model specified by $G$, the maximum likelihood threshold is bounded by:
    \[
    q(G) \leq \mathrm{MLT}(G) \leq \tw(G)+1.
    \]
\end{thm}

In this section, we prove analogous bounds in the extremal setting. The proof of the lower bound is exactly analogous. However, the interpretation of $\Gamma$ as a Euclidean distance matrix allows for a more streamlined proof of the upper bound in the extremal setting.

\begin{thm}
\label{thm:ineqHuesler}
    Let $G$ be a connected graph on $d$ vertices that is not a complete graph.
    In the H\"usler-Reiss graphical model specified by $G$, the extremal maximum likelihood threshold is bounded by:
    \[
    q(G)-1 \leq \emlt(G) \leq \tw(G).
    \]
\end{thm}

\begin{proof}
    Since $G$ is not complete, we have that $q(G) < d$.
    Let $\overline{\Sigma}$ be a generic sample covariance matrix of rank $r$ and let $\overline{\Gamma} = \gamma(\overline{\Sigma})$ so that $\overline{\Gamma}$ is a generic sample variogram of dimensionality $r$. 
    Suppose that $r < q(G) -1$. Let $C$ be a clique in $G$ of size $q(G)$. Consider the $C \times C$ submatrix of $\Sigma$, which has rank at most $q(G) - 2$. So $\overline{\Gamma}_{C,C}$ is not strictly conditionally negative definite by Proposition \ref{prop:LowerRankSigma}. So by Proposition \ref{prop:CNDsubmatrix}, there does not exist a strictly CND matrix $\hat{\Gamma}$ with $\hat{\Gamma}_{C,C} = \overline{\Gamma}_{C,C}$. In other words, this matrix completion problem has no strictly CND solution and the maximum likelihood estimate for $\overline{\Gamma}_{C,C}$ does not exist. Hence $\emlt(G) \geq q(G) -1.$

    Now we turn our attention to the upper bound and assert that $\emlt(G) \leq \tw(G)$.  Let $G'=(V,E')$ be a chordal cover of $G$ of minimal clique number so that $\tw(G) = q(G') -1$. Since $G$ is not complete, $G'$ is also not complete since for any non-edge $f$ of $G$, the graph $([d], \binom{[d]}{2} - \{f\})$ is a chordal cover of $G$. So $\tw(G) < d-1$. Let $\overline{\Gamma}$ be a generic (not necessarily strictly) CND matrix of dimensionality $\tw(G)$. If $\tw(G)+2 \geq d$, then $\overline{\Gamma}$ is strictly CND and we know that its MLE exists. So we assume that $\tw(G) < d-2$. We have that $\overline{\Gamma} = \gamma(\overline{\Sigma})$ for some positive semidefinite $\overline{\Sigma}$ of rank $\tw(G)$.

    We induct on $d$. Since $G'$ is chordal, it is decomposable into $A \cup C$ and $B\cup C$ such that $C$ is a clique that separates $A$ and $B$. Since  $\overline{\Sigma}$ is generic of rank $\tw(G)$, we have that
    \[
    \rank(\overline\Sigma_{A\cup C, A\cup C}) = \min\{\#(A\cup C), \tw(G) \}.
    \]
    We claim that this is at least the treewidth of the induced subgraph $G[A\cup C]$. Indeed, $\#(A\cup C)$ is the number of vertices in this induced subgraph, which is greater than its treewidth. For the second bound, note that for any $S \subset [d]$, $G'[S]$ is a chordal cover of $G[S]$. Let $H$ be a chordal cover of $G[S]$ with minimal clique number. Then $q(H) \leq q(G'[S])$. So 
    \[\tw(G[S]) \leq q(G'[S]) -1 \leq q(G') -1 = \tw(G).
    \]
    Therefore $\rank(\overline\Sigma_{A\cup C, A\cup C}) \geq \tw (G[A \cup C])$ and similarly for $B \cup C$.

    Consider the restrictions of the sample variogram matrices $\overline\Gamma_{A\cup C,A\cup C}$ and $\overline\Gamma_{B\cup C,B\cup C}$.
    By induction, there exist maximum likelihood  estimates $\hat{\Gamma}^A$ and $\hat\Gamma^B$ for these, respectively, that are strictly CND.
    These are both Euclidean distance matrices for  the vertices of simplices $\Delta_{A\cup C} \subset \RR^{A \cup C}$ and $\Delta_{B \cup C} \subset \RR^{B \cup C}$ of dimensions $\#(A \cup C)$ and $\#(B\cup C)$, respectively. Since $C$ is a clique, $\hat\Gamma^A_{C,C} = \overline\Gamma_{C,C} = \hat\Gamma^B_{C,C}.$ So the faces of $\Delta_{A\cup C}$ and $\Delta_{B \cup C}$ corresponding to elements of $C$, denoted $\Delta_{A\cup C}^C$ and $\Delta_{B\cup C}^C$ are isomorphic. We may apply rigid transformations so that 
    \[
    \Delta_{A\cup C}^C \subset \{\bfx \in \RR^{A\cup C} \mid x_i =0 \text{ for all  } i \in A \}, \text{ and}
    \]
    \[
    \Delta_{B\cup C}^C \subset \{\bfy \in \RR^{B\cup C} \mid y_i =0 \text{ for all  } i \in B \},
    \]
    and so that the coordinates of $\Delta_{A\cup C}^C$ and $\Delta_{B\cup C}^C$ indexed by $C$ are equal. 

    By extending the vertices of $\Delta_{A\cup C}$ and $\Delta_{B\cup C}$ by zeros in the appropriate coordinates, we may embed them in $\RR^{A\cup C \cup B}$ so that they agree on the vertices corresponding to elements of $C$. Taking their convex hull results in a simplex $\Delta$ in $\RR^{A\cup C\cup B}$ of dimension $\#(A\cup C\cup B)-1$. The Euclidean distance matrix of $\Delta$, $\Gamma^\Delta$, is therefore strictly CND and has $\Gamma^\Delta_{ij} = \overline\Gamma_{ij}$ for each $(i,j) \in E'$. In particular, it agrees with $\overline\Gamma$ for all edges in $G$. So $\Gamma^\Delta$ is a strictly CND completion of $\overline\Gamma$. By \citet[Proposition~4.6]{HES2022}, since such a completion exists, the MLE exists as well.
\end{proof}

In \citet{uhler2011}, the \emph{elimination criterion} gives an algebraic condition under which the MLE exists with probability 1. We derive an elimination criterion in the extreme setting which involves eliminating variables from a different determinantal ideal. We define this ideal 
\[J_{G,r} \subset \CC[\gamma_{ij} \mid 1 \leq i < j \leq d]\]
as follows. First we define
\[
I_{r} \subset \CC[\sigma_{ij} \mid 1 \leq i \leq j \leq d]
\]
to be the ideal of all $(r+1) \times (r+1)$ minors of a generic symmetric matrix $\Sigma$. Its variety consists of all symmetric matrices of rank at most $r$. Then let $J_{r}$ be obtained from $I_{r}$ via the linear change of coordinates which takes $\Gamma$ to $\sigma(\Gamma)$, see Equation~\ref{eq:sigmaofgamma}. Finally we obtain $J_{G,r}$ from $J_r$ be eliminating all variables $\gamma_{ij}$ such that $(i,j)$ is not an edge of the graph.

\begin{thm}\label{thm:zeroidealcriterion}
Let $r < d$.
    If $J_{G,r}$ is the zero ideal, then the MLE exists with probability 1 for a generic sample variogram matrix of dimensionality $r$.
\end{thm}

\begin{proof}
    Let $\mathcal{V}(J_{G,r-2})$ denote the variety of the ideal $J_{G,r}$. It is the Zariski closure of the projection of $\mathcal{V}(J_{r})$ onto the cone of sufficient statistics $\mathcal{C}_G$, that is, the cone of CND matrices projected on to the entries corresponding to edges of $G$. Let $\dim(\mathcal{V}(J_{G,r})) = k$ and let $\mu$ be a $k$-dimensional Lebesgue measure. We claim that if $\mu(\mathcal{V}(J_{G,r}) \cap \partial \mathcal{C}_G) = 0$, then the MLE exists with probability 1 for a sample variogram matrix of dimensionality $r$. Indeed, let $\overline\Gamma$ be a generic dimensionality $r$ sample variogram matrix. Then  $\sigma(\overline\Gamma)$ has rank $r$ and belongs to $\mathcal{V}(J_{r})$. Since $\mu(\mathcal{V}(J_{G,r}) \cap \partial \mathcal{C}_G) = 0$, its projection onto $\mathcal{C}_G$ is almost surely in the interior of $\mathcal{C}_G$, so it has a strictly CND completion. By \citet[Proposition~4.6]{HES2022}, its MLE exists.

    Now suppose that $J_{G,r}$ is the zero ideal so that its variety is the entire ambient space. Then $V(J_{G,r}) \cap \partial \mathcal{C}_G = \partial \mathcal{C}_G$, which is not full-dimensional. In particular, $\mu(\mathcal{V}(J_{G,r}) \cap \partial \mathcal{C}_G) = 0$ in this case so the MLE for a generic dimensionality $r$ sample variogram matrix exists.
\end{proof}

 Proposition \ref{prop:mltsquare} and Corollary \ref{cor:square} establish the extremal weak maximum likelihood threshold and the extremal maximum likelihood threshold of the 4-cycle (see Example~\ref{ex:4cycle}); in particular, we show that these are different.

\begin{prop}\label{prop:mltsquare}
    The extremal weak maximum likelihood threshold of the 4-cycle is $1$. 
\end{prop}

\begin{proof}
    We show that the probability that the matrix completion problem for the square graph has a solution for a sample with a single element is strictly between $0$ and $1$. 
    For any given graph, 
    scaling all elements in a sample by any nonzero real number
    does not change the property of the empirical variogram matrix to have a conditionally negative definite completion.
    Indeed,
    scaling by a scalar $c$ multiplies the original empirical variogram by the positive number $c^2$.
    Then the original completion multiplied by $c^2$ is a conditionally negative definite completion of the scaled variogram.
    Without loss of generality we therefore assume that we have a sample consisting of a single element $\{(1,x_2,x_3,-(1+x_2+x_3))^\top\}$.

    The variogram matrix arising from this sample is
    \[
    \Gamma = \begin{pmatrix}
        0& \gamma_{12} & (1-x_3)^2 & (2+x_2+x_3)^2 \\
        \gamma_{12} &  0 & (x_2-x_3)^2 & (1+x_3)^2 \\
        (1-x_3)^2 & (x_2-x_3)^2 & 0 & \gamma_{34}\\
        (2+x_2+x_3)^2 & (1+x_3)^2& \gamma_{34} & 0 \end{pmatrix}.
    \]
    From the missing diagonals of the square we obtain two equations 
    arising from the determinants of the corresponding submatrices. 
    Saturating with respect to the determinant of $\CM(\Gamma)$ gives rise to an ideal
    whose Gr\"obner basis for a lexicographic order is then given by the two polynomials
    \begin{small}
    \begin{align*}
        g_1 &= \gamma_{34}^{3}+\frac{-21\,x_2^{2}-4\,x_2\,x_3-4\,x_3^{2}-26\,x_2-4\,x_3-21}{3}\gamma_{34}^{2}+\frac{9\,x_2^{4}+60\,x_2^{3}+118\,\mathit{x_2
      }^{2}+60\,x_2+9}{3}\gamma_{34} \\
      &+\frac{9\,x_2^{6}+36\,x_2^{5}x_3+36\,x_2^{4}x_3^{2}+78\,x_2^{5}+276\,x_2^{4}x_3+240\,x_2^{3
      }x_3^{2}+247\,x_2^{4}+712\,x_2^{3}x_3}{3} \\
      &+\frac{472\,x_2^{2}x_3^{2}+356\,x_2^{3}+712\,x_2^{2}x_3+240\,x_2\,x_3^{2}+247
      \,x_2^{2}+276\,x_2\,x_3+36\,x_3^{2}+78\,x_2+36\,x_3+9}{3} \\
    g_2 &= \gamma_{12}+\frac{-3}{6\,x_2^{2}+20\,x_2+6}\gamma_{34}^{2}+\frac{12\,x_2^{2}+2\,x_2\,x_3+2\,x_3^{2}+18\,x_2+2\,x_3+12}{3\,x_2^{2}+10
      \,x_2+3}\gamma_{34} \\
      &+\frac{-9\,x_2^{2}-4\,x_2\,x_3-4\,x_3^{2}-18\,x_2-4\,x_3-9}{2}.
    \end{align*}
    \end{small}
    The three solutions to this system of equations are the candidates for the completion. 
    A completion is conditionally negative semidefinite exactly when
    the upper left entry, the determinant of the upper $2 \times 2$ block and the whole determinant of the matrix $\Sigma^{(4)}$ 
    are all positive.
    The polynomials in $\gamma_{12}, \gamma_{34}$ 
    corresponding to these conditions are
    \begin{small}
    \begin{align*}
        h_1 &= (x_2+x_3+2)^2, \\
        h_2 &= \frac{-1}{4}(\gamma_{12}-9x_2^2-12x_2x_3-4x_3^2-18x_2-12x_3-9)(\gamma_{12}-x_2^2+2x_2-1),\\
        h_3 &= \frac{-1}{4}(\gamma_{12}^2\gamma_{34}+\gamma_{12}\gamma_{34}^2+(-6x_2^2-4x_2x_3-4x_3^2-8x_2-4x_3-6)\gamma_{12}\gamma_{34}+(3x_2^4+12x_2^3x_3+12x_2^2x_3^2 \\ &+16x_2^3+52x_2^2x_3+40x_2x_3^2+26x_2^2+52x_2x_3+12x_3^2+16x_2+12x_3+3)\gamma_{12}+(3x_2^4-4x_2^3x_3\\    &-4x_2^2x_3^2+4x_2^2x_3+8x_2x_3^2-6x_2^2+4x_2x_3-4x_3^2-4x_3+3)\gamma_{34}-2x_2^6-4x_2^5x_3+12x_2^4x_3^2+32x_2^3x_3^3 \\ &+16x_2^2x_3^4-8x_2^5-20x_2^4x_3-16x_2^3x_3^2-32x_2^2x_3^3-32x_2x_3^4+2x_2^4+24x_2^3x_3+8x_2^2x_3^2-32x_2x_3^3+16x_3^4 \\     &+16x_2^3+24x_2^2x_3-16x_2x_3^2+32x_3^3+2x_2^2-20x_2x_3+12x_3^2-8x_2-4x_3-2).
    \end{align*}
    \end{small}
    
    For the sample consisting of the single element $\{(1,0,2,-3)^\top\}$ we find a completion that satisfies the equations $g_1 = g_2 =0$ and the three inequalities are positive. Since the extremal weak maximum likelihood threshold must be positive, we conclude that it is equal to $1$.
\end{proof}

Now we can apply the elimination criterion in Theorem \ref{thm:zeroidealcriterion} to compute the extremal maximum likelihood threshold of the 4-cycle.

\begin{cor}\label{cor:square}
    The extremal maximum likelihood threshold of the 4-cycle is 2, i.e.~$\emlt(C_4)=2$.
\end{cor}

\begin{proof}
First we show that the extremal maximum likelihood threshold is not 1. We find an open set of dimensionality $1$ partially-specified variogram matrices that do not have a strictly CND completion. 
    Referring back to the equations and inequalities in the proof of Proposition \ref{prop:mltsquare}, we find $\{(1,0,\frac{1}{2},-\frac{3}{2})^\top\}$,
    for which all solutions to $g_1 = g_2 =0$ give $h_3<0$,
    which means that there is a (Euclidean) open region around this sample that has no completion. 
    Thus the extremal maximum likelihood threshold is not equal to $1$.

    Using \texttt{Macaulay2}, we compute that if $G=C_4$, then $J_{G,2}$ is the $0$ ideal. So by Theorem \ref{thm:zeroidealcriterion}, the extremal maximum likelihood threshold of $G$ is at most $2$. Since the extremal maximum likelihood threshold of the 4-cycle is not equal to $1$, we conclude that it is $2$.
\end{proof}

Based on Theorem \ref{thm:ineqHuesler} and Corollary \ref{cor:square}, we conjecture the following relation between the extremal maximum likelihood threshold and the maximum likelihood threshold of a graph.

\begin{conj}
For any undirected graph $G$, it holds that
\[ \emlt(G) = \mathrm{MLT}(G)-1.  \]
\end{conj}

Matroid theory and rigidity theory have proven useful for finding tighter bounds on the maximum likelihood threshold of a Gaussian graphical model. Several papers in the last decade have used rigidity and coloring properties of the graph $G$ to bound, or compute exactly, the Gaussian maximum likelihood threshold of a graph \citep{bernstein2023computing,bernstein2024maximum,gross2018maximum, hojsgaard2024}.  It is a natural direction for future research to approach the extremal maximum likelihood threshold problem from this angle as well.

\section*{Acknowledgements}
This research was supported through the program \emph{Oberwolfach Research Fellows} by the Mathematisches Forschungsinstitut Oberwolfach in 2024.

\bibliographystyle{chicago}
\bibliography{bibliography}

\appendix
\section{Properties of the Variogram Matrix}

Let $\Sigma$ be a positive semidefinite $d\times d$ matrix and let $d_\Sigma$ be the $d$-dimensional vector whose entries are the diagonal of $\Sigma$. We consider the transformation \eqref{eq:gammaofsigma} given by 
$\gamma(\Sigma) = d_\Sigma  \ones^\top  + \ones d_\Sigma - 2\Sigma$.

\begin{prop}\label{prop:RankBoundOnGamma}
    Let $\Sigma \in \S^d_\geq$ be a generic positive semidefinite matrix of rank $r > 0$. Then
    \[
    \rank(\gamma(\Sigma)) = \min(r+2, d). 
    \]
\end{prop}

\begin{proof}
    First note that $\rank(\gamma(\Sigma)) \leq r+2$. Indeed, the expression for $\gamma(\Sigma)$ writes $\gamma(\Sigma)$ as the sum of $\Sigma$, which has rank $r$, and two rank $1$ matrices. So its rank is at most $r+2$.

    For each $r$, we will exhibit an example of a $\Sigma$ of rank $r$ with the vector of all 1's in its kernel such that $\rank(\gamma(\Sigma)) = r+2.$ This will prove the result since satisfying $\rank(\gamma(\Sigma)) = r+2$ is a generic condition on the entries of $\Sigma$. The variety of rank at most $r$ symmetric matrices with $\ones$ in their kernel is irreducible since it is parameterized by the sum of $r$ rank-one matrices of the form $\mathbf{v}\mathbf{v}^\top $ with $\mathbf{v} \cdot \ones = 0$. So it suffices to find any such example of a $\Sigma$ (not necessarily a positive-definite one). Without loss of generality, we may assume $d \leq r+2$. Indeed, if $d > r+2$, appending rows and columns of zeros to the given examples yields the desired examples of the appropriate dimension.

    We start by considering the case where $d = r+2$. We split the proof into two cases based on the parity of $r$. First let $r$ be odd. Let $\tau$ be the permutation of $[r+1]$, 
    
    \[\tau = \prod_{\substack{i = 1 \\ i \text{ odd}}}^r (i, i+1), \]
    
    and let $P_\tau$ be its permutation matrix. Let $\Sigma_r'$ be the $(r+1)\times(r+1)$ matrix,
    \begin{equation}\label{eq:genericrankexample}
    \Sigma_r' := -(r+1)P_\tau + \ones \ones^\top .
    \end{equation}
    Finally let $\Sigma_r$ be the $(r+2)\times (r+2)$ matrix obtained from $\Sigma'_r$ by adding a row and columns of zeros in the $(r+2)$nd row and column.

    Every row of $\Sigma_r$ has one entry equal to $-r$, $r$ entries equal to $1$ and one entry equal to $0$, so $\ones$ is in the kernel of $\Sigma_r$. So $\Sigma_r$ has rank at most $r$. It suffices to show that $\gamma(\Sigma_r)$ has rank $r+2$. The diagonal of $\Sigma$ is the length $r+2$ vector $(1,\dots,1,0)$. So
    \[
    \gamma(\Sigma_r) = \CM(-4(r+1)P_\tau).
    \]
    This is the Cayley-Menger matrix of an invertible $(r+1)\times (r+1)$ matrix. So it has rank $r+2$. 

    Now let $r$ be even. Let $\nu$ be the permutation of $[r+1]$,
    \[\nu = \prod_{\substack{i = 1 \\ i \text{ odd}}}^{r-1} (i, i+1), \]
    which sends $r+1$ to itself. Let $P_\nu$ be its permutation matrix. Let $\Sigma_r'$ be defined analogously to Equation \eqref{eq:genericrankexample} using $P_\nu$ and let $\Sigma_r$ be obtained from $\Sigma_r'$ be appending a row and column of zeros. Then $\Sigma_r$ again is an $(r+2)\times(r+2)$ matrix with $\ones$ and $\bfe_{r+2}$ in its kernel, so it has rank at most $r$. Moreover, $\gamma(\Sigma'_r)$ is again the Cayley-Menger matrix of $2rP_\nu$, an invertible $(r+1) \times (r+1)$ matrix. So $\gamma(\Sigma_r)$ has rank $r+2$, as needed.

    Now we consider the cases where $d = r+1$ or $d=r$. For any $k$, let $M_k$ denote the $k \times k$ matrix whose diagonal entries are $0$ and whose non-diagonal entries are $1$. Let $I_k$ denote the $k \times k$ identity matrix and let $I'_k = \mathrm{diag}(1,\dots, 1,0)$ denote the matrix obtained from $I_k$ be setting the $(k,k)$ entry to $0$.

    Suppose $d = r+1$. Then $\gamma(I'_d) = \CM(-4 M_{d-1})$. Since $-4M_{d-1}$ is an invertible $(d-1)\times(d-1)$ matrix, its Cayley-Menger matrix $\gamma(I'_d)$ is also invertible. Finally if $d = r$, then $\gamma(I_d) = 2M_d$ which is invertible. So for each $r \leq d$, we can find an example of a rank $r$ sample covariance matrix $\Sigma$ such that $\gamma(\Sigma)$ has rank $r+2$, as needed.
\end{proof}

\begin{prop}\label{prop:LowerRankSigma}
    Let $\Sigma$ be an $d \times d$ positive semidefinite matrix of rank $d-2$ with $\ones \in \ker(\Sigma)$. 
    Then $\gamma(\Sigma)$ is not strictly conditionally negative definite. 
    That is, there exists a nonzero $\bfx$ such that $\bfx^\top  \ones = 0$ and $\bfx^\top  \gamma(\Sigma) \bfx = 0.$
\end{prop}

\begin{proof}
    Let $\bfx \in \ker\Sigma$ be a vector which is not a multiple of $\ones$. Since $\ones \in \ker\Sigma$, we may take $\bfx$ to satisfy $\bfx^\top  \ones = 0$. Then
    \begin{align*}
        \bfx^\top  \gamma(\Sigma) \bfx &= \bfx^\top  (-2 \Sigma) \bfx + (\bfx^\top  \ones )(d_\Sigma^\top  \bfx) + (\bfx^\top  d_\Sigma) (\ones^\top  \bfx) =0,
    \end{align*}
    as needed.
\end{proof}

\begin{prop}\label{prop:CNDsubmatrix}
    Let $\Gamma$ be an $d\times d$ strictly conditionally negative definite matrix. Then every principal submatrix of $\Gamma$ is strictly conditionally negative definite.
\end{prop}

\begin{proof}
Let $A \subseteq [d]$. We first note that since $\Gamma \in \mathcal{C}^d$ has zero diagonal, then $\Gamma_{A,A} \in \S_0^d$ too. Now let $y \in \RR^A$, $y\neq \zeros_A$ such that $y \perp \ones_A$. Add 0's to $y$ to obtain $\bar{y} \in \RR^d$. Note that we still have $\bar{y} \perp \ones$, $\bar{y}\neq \zeros$. Therefore
\[ y^\top \Gamma_{A,A} y = \bar{y}^\top \Gamma \bar{y} < 0,
\]
the last inequality holding since $\Gamma$ is strictly CND. Hence $\Gamma_{A,A}$ is also strictly CND.
\end{proof}

\end{document}